\documentclass{article}
\usepackage{amsthm}
\usepackage{amssymb}
\usepackage{amsmath}
\usepackage{graphics}
\usepackage{gastex}
\usepackage[usenames]{color}
\usepackage[mathscr]{eucal}
\newtheorem{theorem}{Theorem}

\newtheorem{lemma}[theorem]{Lemma}
\newtheorem{corollary}[theorem]{Corollary}

\newtheorem{conjecture}[theorem]{Conjecture}

\newenvironment{keywords}{{\bf Keywords: }}{}

\def\maxomega{14}
\def\maxtaula{100}

\renewcommand\leq\leqslant
\renewcommand\geq\geqslant

\title{Ordinarization Transform of a Numerical Semigroup and Semigroups with a Large Number of Intervals}
\author{Maria Bras-Amor\'os}

\begin{document}

\maketitle

\begin{abstract}
A numerical semigroup is said to be ordinary if it has all its gaps in a row.
Indeed, it contains zero and all integers from a given positive one.
One can define a simple operation on a non-ordinary semigroup, which we call here the ordinarization transform, by 
removing its smallest non-zero non-gap (the multiplicity) and adding its largest gap (the Frobenius number).
This gives another numerical semigroup and by repeating this transform several times we end up with an ordinary semigroup. The genus, that is, the number of gaps, is kept constant in all the transforms.

This procedure allows the construction of a tree for each given genus containing all semigrpoups of that genus and rooted in the unique ordinary semigroup of that genus. 
We study here the regularity of these trees and the number of semigroups at each depth. For some depths it is proved that the number of semigroups increases with the genus and it is conjectured that this happens at all given depths.
This may give some lights to a former conjecture saying that the number of semigroups of a given genus increases with the genus.

We finally give an identification between semigroups at a given depth in the ordinarization tree and semigroups with a given (large) number of gap intervals and we give an explicit characterization of those semigroups.
\end{abstract}


\section{Introduction}
Let ${\mathbb N}_0$ 
denote the set of non-negative integers.
A numerical semigroup is a subset of ${\mathbb N}_0$ 
which is closed under addition, contains $0$, and its complement in ${\mathbb N}_0$ is finite.
One main reference for numerical semigroups is \cite{NS}.
For a numerical semigroup $\Lambda$ the elements in ${\mathbb N}_0\setminus\Lambda$ are called {\em gaps} and the number of gaps is the {\em genus} of the semigroup. The largest gap is called the {\em Frobenius number}.
The {\em multiplicity} $m$ of a numerical semigroup is its first non-zero non-gap. 
A numerical semigroup different than ${\mathbb N}_0$ is said to be {\em ordinary} if its gaps are all in a row.

It was conjectured in \cite{Bras:Fibonacci} that the number $n_g$ of numerical semigroups of genus $g$ asymptotically behaves like the Fibonacci numbers. More precisely, it was conjectured that $n_g\geq n_{g-1}+n_{g-2}$, that the limit of the ratio $\frac{n_g}{n_{g-1}+n_{g-2}}$ is $1$ and so that the limit of the ratio
$\frac{n_g}{n_{g-1}}$ is the the golden ratio $\phi=\frac{1+\sqrt{5}}{2}$.
Many other papers deal with the sequence $n_g$ 
\cite{Komeda89,Komeda98,bounds,Anna,Stas,sergi,Zhao,BlancoGarciaPuerto,Kaplan}
and recently Alex Zhai gave a proof for the asymptotic Fibonacci-like behavior of $n_g$ \cite{Zhai}.
However, it has still not been proved that $n_g$ is increasing.
In the present work we approach this problem.

All numerical semigroups can be organized in a tree ${\mathscr T}$
whose root is the semigroup ${\mathbb N}_0$ and in which the parent of a numerical semigroup $\Lambda$ of genus $g$  
is the numerical semigroup obtained by adjoining to $\Lambda$ its Frobenius number. 
So, the parent of a numerical semigroup of genus $g$
has genus $g-1$
and all numerical semigroups are in ${\mathscr T}$, at a depth equal to its genus. In particular, $n_g$ is the number of nodes of ${\mathscr T}$ at depth $g$.
This construction was already considered in 
\cite{RoGaGaJi:oversemigroups}.

Here we will see that all numerical semigroups of a given genus $g$ can be organized in a tree ${\mathscr T}_g$ rooted at the unique ordinary semigroup of genus $g$. 
One significant difference between ${\mathscr T}_g$ and ${\mathscr T}$ is that 
the first one has only a finite number of nodes, indeed, it has $n_g$ nodes, while ${\mathscr T}$ is an infinite tree.
We will see some relations between the trees ${\mathscr T}_g$ and ${\mathscr T}$.
We conjecture that the number of numerical semigroups in 
${\mathscr T}_g$ at a given depth 
is at most the number of numerical semigroups in 
${\mathscr T}_{g+1}$ at the same depth.
This conjecture would prove that $n_{g+1}\geq n_{g}$.
Here we 
show
this result for the lowest and largest depths.

We finally study semigroups with a large number of intervals. We prove that if $\lfloor\frac{n}{2}\rfloor\geq \frac{g+2}{3}$, then the set of numerical semigroups of genus $g$ and $n$ intervals of gaps is empty if $n$ and $g$ have different parity and it is exactly the set of numerical semigroups of genus $g$ and depth $\lfloor\frac{n}{2}\rfloor$ in ${\mathscr {T}}_g$ otherwise.
Furthermore we give an explicit description of the form of these semigroups.

 \section{A tree with all semigroups of a given genus}

\subsection{Ordinarization transform and ordinarization number of a numerical semigroup}

 The {\it ordinarization transform} of a non-ordinary semigroup $\Lambda$ with Frobenius number $F$ and multiplicity $m$ is the set $\Lambda'=\Lambda\setminus\{m\}\cup\{F\}$.
 The ordinarization transform of an ordinary semigroup is itself.
 For instance,
 the ordinarization transform 
 of the semigroup $\Lambda=\{0,{\bf 4},5,8,9,10,12,\dots\}$
 is the semigroup $\Lambda'=\{0,5,8,9,10,{\bf 11},12,\dots\}$.
 Note that the genus of the ordinarization transform of a semigroup is the genus of the semigroup.

 We can iterate the ordinarization transform and set $\Lambda''=(\Lambda')'$
 and in general $\Lambda^{(i)}={\Lambda^{(i-1)}}'$.
 It is easy to check that if $\Lambda$ has genus $g$ 
 then there exists an integer $i$ such that
 $\Lambda^{(i)}=\{0,g+1,g+2,\dots\}$.
 If $\Lambda$ is non-ordinary then there exists a unique $i$ with 
 this property and satisfying 
 $\Lambda^{(i-1)}\neq\{0,g+1,g+2,\dots\}$.
 We call this $i$ the {\it ordinarization number} of $\Lambda$.
 By extension, the ordinarization number of an ordinary semigroup is set to be $0$.  
 For instance,
 if $\Lambda=\{0,{\bf 4},5,8,9,10,12,\dots\}$
 then $\Lambda'=\{0,5,8,9,10,{\bf 11},12,\dots\}$ which is not the ordinary semigroup
 while 
 $\Lambda''=\{0,{\bf 7},8,9,10,11,12,\dots\}$ which is the ordinary 
 semigroup of genus $6$. Thus the ordinarization number of $\Lambda$ is $2$.

\begin{lemma}
\label{lemma:equivdef}
The ordinarization number of a numerical semigroup 
of genus $g$ is the number of its non-zero non-gaps which are smaller than or equal to $g$.
\end{lemma}

\begin{proof}
A numerical semigroup of genus $g$ 
is non-ordinary if and only if its multiplicity 
is at most $g$.
So, we can transform a numerical semigroup while its multiplicity is 
at most $g$. The number of times that we can transform
a semigroup before 
getting the ordinary semigroup is thus the number of its non-zero
non-gaps which are smaller than or equal to the genus.
\end{proof}

Given a numerical semigroup $\Lambda$ it will be convenient to consider its enumeration $\lambda$ as the unique increasing bijective map between ${\mathbb N}_0$ and $\Lambda$. We will use $\lambda_i$ for $\lambda(i)$.
By the previous lemma, if the ordinarization number of a semigroup is $r$ then 
the non-gaps which are at most $g$ are
$\lambda_0=0, \lambda_1, \dots, \lambda_r$.

\begin{lemma}
The maximum ordinarization number of a semigroup of genus $g$ is~$\lfloor\frac{g}{2}\rfloor$.
\end{lemma}

\begin{proof}
Suppose that the ordinarization number of a semigroup is $r$.
On one hand, since the Frobenius number $F$ is at most $2g-1$,
the total number of gaps 
from $1$ to $2g-1$ is $g$
and so the number of non-gaps  
from $1$ to $2g$ is $g$. 
The number of those non-gaps 
which are larger than the genus is $g-r$.
On the other hand 
$\lambda_r+\lambda_1, \lambda_r+\lambda_2, \dots, 2\lambda_r$
are different non-gaps between $g+1$ and $2g$.
So, the number of non-gaps between $g+1$ and $2g$ is at least $r$.
Putting this altogether we get that $g-r\geq r$ and so $r\leq\frac{g}{2}$.

On the other hand, the bound is attained by 
the
semigroup 
\begin{equation}
\{0,2,4,\dots,2\left\lfloor\frac{g}{2}\right\rfloor,2\left(\left\lfloor\frac{g}{2}\right\rfloor +1\right),\dots,2g,2g+1,2g+2,\dots\}.
\label{hyperelliptic}
\end{equation}
\end{proof}

We want to see that the maximum ordinarization number as stated in the previous lemma is attained only by 
the semigroup in (\ref{hyperelliptic}).
For this purpose we need the next lemma. Its proof has been omitted but can easily be obtained.

 \begin{lemma}
 \label{lemma:arithseqsums}
 Consider a finite subset $A=\{a_1<\dots<a_n\}\subseteq {\mathbb N}_0$.
 \begin{enumerate}
 \item
 The set $A+A$ contains at least $2n-1$ elements
 \item
 The set $A+A$ contains exactly $2n-1$ elements if and only if 
 $a_{i}=a_1+\alpha i$ for all $i$ and for a given positive integer $\alpha$.
 \end{enumerate}
 \end{lemma}

 \begin{lemma}
 Let $g>0$.
 The unique numerical semigroup of genus $g$ and 
 ordinarization number $\lfloor\frac{g}{2}\rfloor$ is $\{0,2,4,\dots,2g,2g+1,2g+2,\dots\}$.
 \end{lemma}

\begin{proof}
Suppose that the ordinarization number of $\Lambda$ is $\lfloor\frac{g}{2}\rfloor$.
Since $\lambda_{\lfloor\frac{g}{2}\rfloor}\leq g$,
we know that the set of all non-gaps between $0$ and $2g$ 
must contain all the sums 
$\{\lambda_i+\lambda_j:0\leq i,j\leq \lfloor\frac{g}{2}\rfloor\}$.
But the number of non-gaps between $0$ and $2g$ is $g+1$,
while by 
Lemma~\ref{lemma:arithseqsums} 
the set of sums above
has at least $2\lfloor\frac{g}{2}\rfloor+1$ elements.

If $g$ is even then 
by 
the second item of
Lemma~\ref{lemma:arithseqsums}
we get that 
$\lambda_i=i\lambda_1$ 
for $i\leq\frac{g}{2}$.
Now $\lambda_{\frac{g}{2}}= 
\frac{g}{2}\lambda_1\leq g$ means that
$\lambda_1\leq 2$. If $\lambda_1=1$ this contradicts
$g>0$.
So,
$\lambda_i=2i$ 
for $0\leq i\leq\frac{g}{2}$ and the 
remaining non-gaps 
between $g+1$ and $2g$ are necessarily
$\lambda_i=2i$ 
for $i=\frac{g}{2}+1$ to $i=g$.

If $g$ is odd and $g\not\in\Lambda$ then 
we know that the set of all non-gaps between $0$ and $2g-1$ 
must contain all the sums 
$\{\lambda_i+\lambda_j:0\leq i,j\leq \lfloor\frac{g}{2}\rfloor\}$.
But the number of non-gaps between $0$ and $2g-1$ is $g$,
while by 
Lemma~\ref{lemma:arithseqsums} 
the set of sums above
has at least $2\lfloor\frac{g}{2}\rfloor+1=g$ elements
and we can argue as before.
\end{proof}

\subsection{The tree \boldmath{${\mathscr T}_g$}}

 The definition of the ordinarization transform of a numerical semigroup allows the construction of a tree
 ${\mathscr T}_g$ on the set of all numerical semigroups of a given genus rooted at the unique ordinary semigroup of this genus, 
 where the parent of a semigroup is its ordinarization transform and the descendants of a semigroup are the semigroups obtained by taking away 
 a generator larger than the Frobenius number and adding a new non-gap smaller than the multiplicity in a licit place.
To illustrate this construction with an example 
in Figure~\ref{fig:t6} we depicted ${\mathscr T}_6$.
 The depth of a numerical semigroup of genus $g$ in ${\mathscr T}_g$ is then its ordinarization number.

\begin{figure}
\begin{center}
\compatiblegastexun
\setvertexdiam{2}
\letvertex O0=(10,170)
\letvertex O1=(50,170)
\letvertex O2=(90,170)
\letvertex O3=(130,170)
\letvertex A0=(10,85)
\letvertex A0cua=(24,85)
\letvertex A101=(50,10)
\letvertex A101cap=(36,10)
\letvertex A102=(50,20)
\letvertex A102cua=(64,20)
\letvertex A102cap=(36,20)
\letvertex A103=(50,50)
\letvertex A103cua=(64,50)
\letvertex A103cap=(36,50)
\letvertex A104=(50,70)
\letvertex A104cap=(36,70)
\letvertex A105=(50,90)
\letvertex A105cap=(36,90)
\letvertex A105cua=(64,90)
\letvertex A106=(50,100)
\letvertex A106cap=(36,100)
\letvertex A107=(50,110)
\letvertex A107cap=(36,110)
\letvertex A108=(50,120)
\letvertex A108cap=(36,120)
\letvertex A109=(50,130)
\letvertex A109cap=(36,130)
\letvertex A110=(50,140)
\letvertex A110cap=(36,140)
\letvertex A111=(50,150)
\letvertex A111cap=(36,150)
\letvertex A112=(50,160)
\letvertex A112cap=(36,160)
\letvertex A21=(90,15)
\letvertex A21cap=(76,15)
\letvertex A22=(90,25)
\letvertex A22cap=(76,25)
\letvertex A23=(90,35)
\letvertex A23cap=(76,35)
\letvertex A24=(90,45)
\letvertex A24cap=(76,45)
\letvertex A24cua=(104,45)
\letvertex A25=(90,55)
\letvertex A25cap=(76,55)
\letvertex A26=(90,65)
\letvertex A26cap=(76,65)
\letvertex A27=(90,75)
\letvertex A27cap=(76,75)
\letvertex A28=(90,85)
\letvertex A28cap=(76,85)
\letvertex A29=(90,95)
\letvertex A29cap=(76,95)
\letvertex A3=(130,45)
\letvertex A3cap=(116,45)
\resizebox{\textwidth}{!}{
\begin{picture}(150,150)(-10,10)
\drawvertex(O0){\makebox(0,0){\scalebox{1.5}{ord. number $0$}}}
\drawvertex(O1){\makebox(0,0){\scalebox{1.5}{ord. number $1$}}}
\drawvertex(O2){\makebox(0,0){\scalebox{1.5}{ord. number $2$}}}
\drawvertex(O3){\makebox(0,0){\scalebox{1.5}{ord. number $3$}}}
{\tiny
\drawvertex(A0){\makebox(0,0){$\{0,7,8,9,10,11,12,\dots\}$}}
\node[fillcolor=White](a00)(-4,85){0}
\node[fillcolor=Gray](a01)(-2,85){}
\node[fillcolor=Gray](a02)(0,85){}
\node[fillcolor=Gray](a03)(2,85){}
\node[fillcolor=Gray](a04)(4,85){}
\node[fillcolor=Gray](a05)(6,85){}
\node[fillcolor=Gray](a06)(8,85){}
\node[fillcolor=White](a07)(10,85){7}
\node[fillcolor=White](a08)(12,85){8}
\node[fillcolor=White](a09)(14,85){9}
\node[fillcolor=White](a010)(16,85){10}
\node[fillcolor=White](a011)(18,85){11}
\node[fillcolor=White](a012)(20,85){12}
\node[fillcolor=White](a013)(22,85){13}
\node[fillcolor=White](a014)(24,85){\dots}

\drawvertex(A101){\makebox(0,0){$\{0, 4, 8, 9, 10, 11, 12,\dots\}$}} 
\node[fillcolor=White](a00)(36,10){0}
\node[fillcolor=Gray](a1011)(38,10){}
\node[fillcolor=Gray](a1012)(40,10){}
\node[fillcolor=Gray](a1013)(42,10){}
\node[fillcolor=White](a1014)(44,10){4}
\node[fillcolor=Gray](a1015)(46,10){}
\node[fillcolor=Gray](a1016)(48,10){}
\node[fillcolor=Gray](a1017)(50,10){}
\node[fillcolor=White](a1018)(52,10){8}
\node[fillcolor=White](a1019)(54,10){9}
\node[fillcolor=White](a10110)(56,10){10}
\node[fillcolor=White](a10111)(58,10){11}
\node[fillcolor=White](a10112)(60,10){12}
\node[fillcolor=White](a10113)(62,10){13}
\node[fillcolor=White](a10114)(64,10){\dots}

\drawvertex(A102){\makebox(0,0){$\{0, 5, 8, 9, 10, 11, 12,\dots\}$}}
\node[fillcolor=White](a)(36,20){0}
\node[fillcolor=Gray](a)(38,20){}
\node[fillcolor=Gray](a)(40,20){}
\node[fillcolor=Gray](a)(42,20){}
\node[fillcolor=Gray](a)(44,20){}
\node[fillcolor=White](a)(46,20){5}
\node[fillcolor=Gray](a)(48,20){}
\node[fillcolor=Gray](a)(50,20){}
\node[fillcolor=White](a)(52,20){8}
\node[fillcolor=White](a)(54,20){9}
\node[fillcolor=White](a)(56,20){10}
\node[fillcolor=White](a)(58,20){11}
\node[fillcolor=White](a)(60,20){12}
\node[fillcolor=White](a)(62,20){13}
\node[fillcolor=White](a)(64,20){\dots}

\drawvertex(A103){\makebox(0,0){$\{0, 6, 8, 9, 10, 11, 12,\dots\}$}} 
\node[fillcolor=White](a)(36,50){0}
\node[fillcolor=Gray](a)(38,50){}
\node[fillcolor=Gray](a)(40,50){}
\node[fillcolor=Gray](a)(42,50){}
\node[fillcolor=Gray](a)(44,50){}
\node[fillcolor=Gray](a)(46,50){}
\node[fillcolor=White](a)(48,50){6}
\node[fillcolor=Gray](a)(50,50){}
\node[fillcolor=White](a)(52,50){8}
\node[fillcolor=White](a)(54,50){9}
\node[fillcolor=White](a)(56,50){10}
\node[fillcolor=White](a)(58,50){11}
\node[fillcolor=White](a)(60,50){12}
\node[fillcolor=White](a)(62,50){13}
\node[fillcolor=White](a)(64,50){\dots}

\drawvertex(A104){\makebox(0,0){$\{0, 5, 7, 9, 10, 11, 12,\dots\}$}}
\node[fillcolor=White](a)(36,70){0}
\node[fillcolor=Gray](a)(38,70){}
\node[fillcolor=Gray](a)(40,70){}
\node[fillcolor=Gray](a)(42,70){}
\node[fillcolor=Gray](a)(44,70){}
\node[fillcolor=White](a)(46,70){5}
\node[fillcolor=Gray](a)(48,70){}
\node[fillcolor=White](a)(50,70){7}
\node[fillcolor=Gray](a)(52,70){}
\node[fillcolor=White](a)(54,70){9}
\node[fillcolor=White](a)(56,70){10}
\node[fillcolor=White](a)(58,70){11}
\node[fillcolor=White](a)(60,70){12}
\node[fillcolor=White](a)(62,70){13}
\node[fillcolor=White](a)(64,70){\dots}

\drawvertex(A105){\makebox(0,0){$\{0, 6, 7, 9, 10, 11, 12,\dots\}$}} 
\node[fillcolor=White](a)(36,90){0}
\node[fillcolor=Gray](a)(38,90){}
\node[fillcolor=Gray](a)(40,90){}
\node[fillcolor=Gray](a)(42,90){}
\node[fillcolor=Gray](a)(44,90){}
\node[fillcolor=Gray](a)(46,90){}
\node[fillcolor=White](a)(48,90){6}
\node[fillcolor=White](a)(50,90){7}
\node[fillcolor=Gray](a)(52,90){}
\node[fillcolor=White](a)(54,90){9}
\node[fillcolor=White](a)(56,90){10}
\node[fillcolor=White](a)(58,90){11}
\node[fillcolor=White](a)(60,90){12}
\node[fillcolor=White](a)(62,90){13}
\node[fillcolor=White](a)(64,90){\dots}

\drawvertex(A106){\makebox(0,0){$\{0, 4, 7, 8, 10, 11, 12,\dots\}$}} 
\node[fillcolor=White](a)(36,100){0}
\node[fillcolor=Gray](a)(38,100){}
\node[fillcolor=Gray](a)(40,100){}
\node[fillcolor=Gray](a)(42,100){}
\node[fillcolor=White](a)(44,100){4}
\node[fillcolor=Gray](a)(46,100){}
\node[fillcolor=Gray](a)(48,100){}
\node[fillcolor=White](a)(50,100){7}
\node[fillcolor=White](a)(52,100){8}
\node[fillcolor=Gray](a)(54,100){}
\node[fillcolor=White](a)(56,100){10}
\node[fillcolor=White](a)(58,100){11}
\node[fillcolor=White](a)(60,100){12}
\node[fillcolor=White](a)(62,100){13}
\node[fillcolor=White](a)(64,100){\dots}

\drawvertex(A107){\makebox(0,0){$\{0, 5, 7, 8, 10, 11, 12,\dots\}$}}
\node[fillcolor=White](a)(36,110){0}
\node[fillcolor=Gray](a)(38,110){}
\node[fillcolor=Gray](a)(40,110){}
\node[fillcolor=Gray](a)(42,110){}
\node[fillcolor=Gray](a)(44,110){}
\node[fillcolor=White](a)(46,110){5}
\node[fillcolor=Gray](a)(48,110){}
\node[fillcolor=White](a)(50,110){7}
\node[fillcolor=White](a)(52,110){8}
\node[fillcolor=Gray](a)(54,110){}
\node[fillcolor=White](a)(56,110){10}
\node[fillcolor=White](a)(58,110){11}
\node[fillcolor=White](a)(60,110){12}
\node[fillcolor=White](a)(62,110){13}
\node[fillcolor=White](a)(64,110){\dots}

\drawvertex(A108){\makebox(0,0){$\{0, 6, 7, 8, 10, 11, 12,\dots\}$}}
\node[fillcolor=White](a)(36,120){0}
\node[fillcolor=Gray](a)(38,120){}
\node[fillcolor=Gray](a)(40,120){}
\node[fillcolor=Gray](a)(42,120){}
\node[fillcolor=Gray](a)(44,120){}
\node[fillcolor=Gray](a)(46,120){}
\node[fillcolor=White](a)(48,120){6}
\node[fillcolor=White](a)(50,120){7}
\node[fillcolor=White](a)(52,120){8}
\node[fillcolor=Gray](a)(54,120){}
\node[fillcolor=White](a)(56,120){10}
\node[fillcolor=White](a)(58,120){11}
\node[fillcolor=White](a)(60,120){12}
\node[fillcolor=White](a)(62,120){13}
\node[fillcolor=White](a)(64,120){\dots}

\drawvertex(A109){\makebox(0,0){$\{0, 4, 7, 8, 9, 11, 12,\dots\}$}}
\node[fillcolor=White](a)(36,130){0}
\node[fillcolor=Gray](a)(38,130){}
\node[fillcolor=Gray](a)(40,130){}
\node[fillcolor=Gray](a)(42,130){}
\node[fillcolor=White](a)(44,130){4}
\node[fillcolor=Gray](a)(46,130){}
\node[fillcolor=Gray](a)(48,130){}
\node[fillcolor=White](a)(50,130){7}
\node[fillcolor=White](a)(52,130){8}
\node[fillcolor=White](a)(54,130){9}
\node[fillcolor=Gray](a)(56,130){}
\node[fillcolor=White](a)(58,130){11}
\node[fillcolor=White](a)(60,130){12}
\node[fillcolor=White](a)(62,130){13}
\node[fillcolor=White](a)(64,130){\dots}

\drawvertex(A110){\makebox(0,0){$\{0, 6, 7, 8, 9, 11, 12,\dots\}$}}
\node[fillcolor=White](a)(36,140){0}
\node[fillcolor=Gray](a)(38,140){}
\node[fillcolor=Gray](a)(40,140){}
\node[fillcolor=Gray](a)(42,140){}
\node[fillcolor=Gray](a)(44,140){}
\node[fillcolor=Gray](a)(46,140){}
\node[fillcolor=White](a)(48,140){6}
\node[fillcolor=White](a)(50,140){7}
\node[fillcolor=White](a)(52,140){8}
\node[fillcolor=White](a)(54,140){9}
\node[fillcolor=Gray](a)(56,140){}
\node[fillcolor=White](a)(58,140){11}
\node[fillcolor=White](a)(60,140){12}
\node[fillcolor=White](a)(62,140){13}
\node[fillcolor=White](a)(64,140){\dots}

\drawvertex(A111){\makebox(0,0){$\{0, 5, 7, 8, 9, 10, 12,\dots\}$}}
\node[fillcolor=White](a)(36,150){0}
\node[fillcolor=Gray](a)(38,150){}
\node[fillcolor=Gray](a)(40,150){}
\node[fillcolor=Gray](a)(42,150){}
\node[fillcolor=Gray](a)(44,150){}
\node[fillcolor=White](a)(46,150){5}
\node[fillcolor=Gray](a)(48,150){}
\node[fillcolor=White](a)(50,150){7}
\node[fillcolor=White](a)(52,150){8}
\node[fillcolor=White](a)(54,150){9}
\node[fillcolor=White](a)(56,150){10}
\node[fillcolor=Gray](a)(58,150){}
\node[fillcolor=White](a)(60,150){12}
\node[fillcolor=White](a)(62,150){13}
\node[fillcolor=White](a)(64,150){\dots}

\drawvertex(A112){\makebox(0,0){$\{0, 6, 7, 8, 9, 10, 12,\dots\}$}} 
\node[fillcolor=White](a)(36,160){0}
\node[fillcolor=Gray](a)(38,160){}
\node[fillcolor=Gray](a)(40,160){}
\node[fillcolor=Gray](a)(42,160){}
\node[fillcolor=Gray](a)(44,160){}
\node[fillcolor=Gray](a)(46,160){}
\node[fillcolor=White](a)(48,160){6}
\node[fillcolor=White](a)(50,160){7}
\node[fillcolor=White](a)(52,160){8}
\node[fillcolor=White](a)(54,160){9}
\node[fillcolor=White](a)(56,160){10}
\node[fillcolor=Gray](a)(58,160){}
\node[fillcolor=White](a)(60,160){12}
\node[fillcolor=White](a)(62,160){13}
\node[fillcolor=White](a)(64,160){\dots}
%
\drawvertex(A21){\makebox(0,0){$\{0, 4, 5, 8, 9, 10, 12,\dots\}$}}
\node[fillcolor=White](a)(76,15){0}
\node[fillcolor=Gray](a)(78,15){}
\node[fillcolor=Gray](a)(80,15){}
\node[fillcolor=Gray](a)(82,15){}
\node[fillcolor=White](a)(84,15){4}
\node[fillcolor=White](a)(86,15){5}
\node[fillcolor=Gray](a)(88,15){}
\node[fillcolor=Gray](a)(90,15){}
\node[fillcolor=White](a)(92,15){8}
\node[fillcolor=White](a)(94,15){9}
\node[fillcolor=White](a)(96,15){10}
\node[fillcolor=Gray](a)(98,15){}
\node[fillcolor=White](a)(100,15){12}
\node[fillcolor=White](a)(102,15){13}
\node[fillcolor=White](a)(104,15){\dots}

\drawvertex(A22){\makebox(0,0){$\{0, 3, 6, 9, 10, 11, 12,\dots\}$}}
\node[fillcolor=White](a)(76,25){0}
\node[fillcolor=Gray](a)(78,25){}
\node[fillcolor=Gray](a)(80,25){}
\node[fillcolor=White](a)(82,25){3}
\node[fillcolor=Gray](a)(84,25){}
\node[fillcolor=Gray](a)(86,25){}
\node[fillcolor=White](a)(88,25){6}
\node[fillcolor=Gray](a)(90,25){}
\node[fillcolor=Gray](a)(92,25){}
\node[fillcolor=White](a)(94,25){9}
\node[fillcolor=White](a)(96,25){10}
\node[fillcolor=White](a)(98,25){11}
\node[fillcolor=White](a)(100,25){12}
\node[fillcolor=White](a)(102,25){13}
\node[fillcolor=White](a)(104,25){\dots}

\drawvertex(A23){\makebox(0,0){$\{0, 5, 6, 9, 10, 11, 12,\dots\}$}}
\node[fillcolor=White](a)(76,35){0}
\node[fillcolor=Gray](a)(78,35){}
\node[fillcolor=Gray](a)(80,35){}
\node[fillcolor=Gray](a)(82,35){}
\node[fillcolor=Gray](a)(84,35){}
\node[fillcolor=White](a)(86,35){5}
\node[fillcolor=White](a)(88,35){6}
\node[fillcolor=Gray](a)(90,35){}
\node[fillcolor=Gray](a)(92,35){}
\node[fillcolor=White](a)(94,35){9}
\node[fillcolor=White](a)(96,35){10}
\node[fillcolor=White](a)(98,35){11}
\node[fillcolor=White](a)(100,35){12}
\node[fillcolor=White](a)(102,35){13}
\node[fillcolor=White](a)(104,35){\dots}

\drawvertex(A24){\makebox(0,0){$\{0, 4, 6, 8, 10, 11, 12,\dots\}$}}
\node[fillcolor=White](a)(76,45){0}
\node[fillcolor=Gray](a)(78,45){}
\node[fillcolor=Gray](a)(80,45){}
\node[fillcolor=Gray](a)(82,45){}
\node[fillcolor=White](a)(84,45){4}
\node[fillcolor=Gray](a)(86,45){}
\node[fillcolor=White](a)(88,45){6}
\node[fillcolor=Gray](a)(90,45){}
\node[fillcolor=White](a)(92,45){8}
\node[fillcolor=Gray](a)(94,45){}
\node[fillcolor=White](a)(96,45){10}
\node[fillcolor=White](a)(98,45){11}
\node[fillcolor=White](a)(100,45){12}
\node[fillcolor=White](a)(102,45){13}
\node[fillcolor=White](a)(104,45){\dots}

\drawvertex(A25){\makebox(0,0){$\{0, 5, 6, 8, 10, 11, 12,\dots\}$}}
\node[fillcolor=White](a)(76,55){0}
\node[fillcolor=Gray](a)(78,55){}
\node[fillcolor=Gray](a)(80,55){}
\node[fillcolor=Gray](a)(82,55){}
\node[fillcolor=Gray](a)(84,55){}
\node[fillcolor=White](a)(86,55){5}
\node[fillcolor=White](a)(88,55){6}
\node[fillcolor=Gray](a)(90,55){}
\node[fillcolor=White](a)(92,55){8}
\node[fillcolor=Gray](a)(94,55){}
\node[fillcolor=White](a)(96,55){10}
\node[fillcolor=White](a)(98,55){11}
\node[fillcolor=White](a)(100,55){12}
\node[fillcolor=White](a)(102,55){13}
\node[fillcolor=White](a)(104,55){\dots}

\drawvertex(A26){\makebox(0,0){$\{0, 3, 6, 8, 9, 11, 12,\dots\}$}}
\node[fillcolor=White](a)(76,65){0}
\node[fillcolor=Gray](a)(78,65){}
\node[fillcolor=Gray](a)(80,65){}
\node[fillcolor=White](a)(82,65){3}
\node[fillcolor=Gray](a)(84,65){}
\node[fillcolor=Gray](a)(86,65){}
\node[fillcolor=White](a)(88,65){6}
\node[fillcolor=Gray](a)(90,65){}
\node[fillcolor=White](a)(92,65){8}
\node[fillcolor=White](a)(94,65){9}
\node[fillcolor=Gray](a)(96,65){}
\node[fillcolor=White](a)(98,65){11}
\node[fillcolor=White](a)(100,65){12}
\node[fillcolor=White](a)(102,65){13}
\node[fillcolor=White](a)(104,65){\dots}

\drawvertex(A27){\makebox(0,0){$\{0, 4, 6, 8, 9, 10, 12,\dots\}$}}
\node[fillcolor=White](a)(76,75){0}
\node[fillcolor=Gray](a)(78,75){}
\node[fillcolor=Gray](a)(80,75){}
\node[fillcolor=Gray](a)(82,75){}
\node[fillcolor=White](a)(84,75){4}
\node[fillcolor=Gray](a)(86,75){}
\node[fillcolor=White](a)(88,75){6}
\node[fillcolor=Gray](a)(90,75){}
\node[fillcolor=White](a)(92,75){8}
\node[fillcolor=White](a)(94,75){9}
\node[fillcolor=White](a)(96,75){10}
\node[fillcolor=Gray](a)(98,75){}
\node[fillcolor=White](a)(100,75){12}
\node[fillcolor=White](a)(102,75){13}
\node[fillcolor=White](a)(104,75){\dots}

\drawvertex(A28){\makebox(0,0){$\{0, 3, 6, 7, 9, 10, 12,\dots\}$}} 
\node[fillcolor=White](a)(76,85){0}
\node[fillcolor=Gray](a)(78,85){}
\node[fillcolor=Gray](a)(80,85){}
\node[fillcolor=White](a)(82,85){3}
\node[fillcolor=Gray](a)(84,85){}
\node[fillcolor=Gray](a)(86,85){}
\node[fillcolor=White](a)(88,85){6}
\node[fillcolor=White](a)(90,85){7}
\node[fillcolor=Gray](a)(92,85){}
\node[fillcolor=White](a)(94,85){9}
\node[fillcolor=White](a)(96,85){10}
\node[fillcolor=Gray](a)(98,85){}
\node[fillcolor=White](a)(100,85){12}
\node[fillcolor=White](a)(102,85){13}
\node[fillcolor=White](a)(104,85){\dots}

\drawvertex(A29){\makebox(0,0){$\{0, 5, 6, 7, 10, 11, 12\dots\}$}} 
\node[fillcolor=White](a)(76,95){0}
\node[fillcolor=Gray](a)(78,95){}
\node[fillcolor=Gray](a)(80,95){}
\node[fillcolor=Gray](a)(82,95){}
\node[fillcolor=Gray](a)(84,95){}
\node[fillcolor=White](a)(86,95){5}
\node[fillcolor=White](a)(88,95){6}
\node[fillcolor=White](a)(90,95){7}
\node[fillcolor=Gray](a)(92,95){}
\node[fillcolor=Gray](a)(94,95){}
\node[fillcolor=White](a)(96,95){10}
\node[fillcolor=White](a)(98,95){11}
\node[fillcolor=White](a)(100,95){12}
\node[fillcolor=White](a)(102,95){13}
\node[fillcolor=White](a)(104,95){\dots}

\drawvertex(A3){\makebox(0,0){$\{0,2,4,6,8,10,12,\dots\}$}}
\node[fillcolor=White](a)(116,45){0}
\node[fillcolor=Gray](a)(118,45){}
\node[fillcolor=White](a)(120,45){2}
\node[fillcolor=Gray](a)(122,45){}
\node[fillcolor=White](a)(124,45){4}
\node[fillcolor=Gray](a)(126,45){}
\node[fillcolor=White](a)(128,45){6}
\node[fillcolor=Gray](a)(130,45){}
\node[fillcolor=White](a)(132,45){8}
\node[fillcolor=Gray](a)(134,45){}
\node[fillcolor=White](a)(136,45){10}
\node[fillcolor=Gray](a)(138,45){}
\node[fillcolor=White](a)(140,45){12}
\node[fillcolor=White](a)(142,45){13}
\node[fillcolor=White](a)(144,45){\dots}

\drawundirectededge(A0cua,A101cap){}
\drawundirectededge(A0cua,A102cap){}
\drawundirectededge(A0cua,A103cap){}
\drawundirectededge(A0cua,A104cap){}
\drawundirectededge(A0cua,A105cap){}
\drawundirectededge(A0cua,A106cap){}
\drawundirectededge(A0cua,A107cap){}
\drawundirectededge(A0cua,A108cap){}
\drawundirectededge(A0cua,A109cap){}
\drawundirectededge(A0cua,A110cap){}
\drawundirectededge(A0cua,A111cap){}
\drawundirectededge(A0cua,A112cap){}
\drawundirectededge(A21cap,A102cua){}
\drawundirectededge(A22cap,A103cua){}
\drawundirectededge(A23cap,A103cua){}
\drawundirectededge(A24cap,A103cua){}
\drawundirectededge(A25cap,A103cua){}
\drawundirectededge(A26cap,A103cua){}
\drawundirectededge(A27cap,A103cua){}
\drawundirectededge(A28cap,A105cua){}
\drawundirectededge(A29cap,A105cua){}
\drawundirectededge(A24cua,A3cap){}
}
\end{picture}
}
\caption{${\mathscr T}_6$}
\label{fig:t6}
\end{center}
\end{figure}

The next two lemmas show some relations between ${\mathscr T}$ and ${\mathscr T}_g$. 

\begin{lemma}
If $\Lambda_1$ is a descendant of $\Lambda_2$ in ${\mathscr T}$ then $\Lambda_1'$ is a descendant of $\Lambda_2'$ in~${\mathscr T}$.
\end{lemma}

\begin{proof}
This is obvious when $\Lambda_1$ (and so $\Lambda_2$) is an ordinary semigroup.
Suppose that $\Lambda_1$ is not ordinary. It is easy to check that in this case 
$\Lambda_1$ and $\Lambda_2$ have the same multiplicity $m$ and that
$\Lambda_1\setminus\{m\}$ 
is a descendant of 
$\Lambda_2\setminus\{m\}$ in ${\mathscr T}$.
Now the lemma is a consequence of the fact that if we adjoin to a semigroup 
(in our case $\Lambda_1\setminus\{m\}$)
its Frobenius number we obtain its parent in ${\mathscr T}$ 
(in our case $\Lambda_2\setminus\{m\}=\Lambda_1'$)
and if we repeat the same procedure 
(in our case obtaining $\Lambda_2'$)
we obtain the parent of the parent in ${\mathscr T}$ 
(in our case the parent of $\Lambda_1'$ in ${\mathscr T}$).
\end{proof}

\begin{lemma}
If two non-ordinary semigroups $\Lambda_1$ and $\Lambda_2$ with the same genus $g$ have the same parent in ${\mathscr T}$ then they also have the same parent in ${\mathscr T}_g$.
\end{lemma}

\begin{proof}
Indeed, the parent of $\Lambda_1$ and $\Lambda_2$ in 
${\mathscr T}_g$ is the parent they have in 
${\mathscr T}$ without its multiplicity.
\end{proof}

\subsection{Conjecture}

 Let $n_{g,r}$ be the number of numerical semigroups of genus 
 $g$ and ordinarization number $r$.
 For each genus $g\leq 49$
 we computed $n_{g,r}$ for each ordinarization number $r$ from $0$ up to 
 $\lfloor\frac{g}{2}\rfloor$. The results are given in Table~\ref{t:experimentalresults}. One can observe that for each ordinarization number the number of semigroups of this ordinarization number increases with the genus or stays the same.
 By extending the definition of $n_{g,r}$ for $r>\lfloor\frac{g}{2}\rfloor$
 by setting $n_{g,r}=0$ in this case, this leads to the next conjecture.

 \begin{conjecture}
 For each genus $g\in{\mathbb N}_0$
 and each ordinarization number $r\in{\mathbb N}_0$, $n_{g,r}\leq n_{g+1,r}$.
 \end{conjecture}

 This is equivalent to the number of numerical semigroups at a given depth 
 of ${\mathscr T}_g$
 being at most the 
 number of numerical semigroups at the same depth of ${\mathscr T}_{g+1}$.
 If the conjecture were true then the total number of nodes in
 ${\mathscr T}_g$ would be at most the total number of nodes in
 ${\mathscr T}_{g+1}$ proving that $n_g$ increases with~$g$.

 \section{Partial proofs of the conjecture}
 We will prove the conjecture for particular values of the pair $g,r$.
 We wrote these values in bold face in Table~\ref{t:experimentalresults}.
 It is obvious that for $r=0$ we always have $n_{g,r}=1$ since 
 for any genus the ordinary semigroup is the unique numerical semigroup of 
 ordinarization number $0$.
 In the next subsections we will prove the conjecture for 
 $n_{g,1}$ and any $g$ and for $n_{g,r}$ and any $g$, whenever $r\geq\max(\frac{g}{3}+1,\lfloor\frac{g+1}{2}\rfloor-\maxomega)$. 

 \subsection{Ordinarization number 1}
 \begin{lemma}
 \label{lemma:casu}
 Let $g\in{\mathbb N}_0$.
 The number of semigroups of genus $g$ and ordinarization number $1$ is 
 $$n_{g,1}=
 \left\lceil\frac{g-1}{2}\right\rceil\left\lfloor\frac{g+1}{2}\right\rfloor
 +
 \frac{\lfloor\frac{g-1}{2}\rfloor\lfloor\frac{g+1}{2}\rfloor}{2}
 =
 \left\{
 \begin{array}{ll}
 \frac{3}{8}g^2-\frac{1}{4}g & \mbox{ if }g\mbox{ is even,}\\
 \frac{3}{8}g^2-\frac{3}{8} & \mbox{ if }g\mbox{ is odd.}
 \end{array}
 \right.
 $$
 \end{lemma}

\begin{proof}
The semigroups of ordinarization number one are obtained 
from the ordinary semigroup $\{0,g+1,g+2,\dots\}$
by taking out one non-gap $a$ and adding a non-zero non-gap $b$ 
smaller than $g+1$.
Since any element in the ordinary semigroup larger than $2g+1$ 
is a sum of two non-zero non-gaps it can not be taken out.
So,  
$$
g+1\leq a\leq 2g+1.
$$

Fix $a$ in the previous range. For $b$ we have four 
necessary conditions which together become sufficient:
\begin{enumerate}
\item $b\leq g$ since $b$ must be a gap of the ordinary semigroup;
\item $a-b\leq g$ since otherwise $a=b+(a-b)$ and so $a$ 
must be in the new semigroup;
\item $2b\geq g+1$ because otherwise $b,2b$ are two different non-gaps which are at most $g$ contradicting that the ordinarization number of the new semigroup is $1$;
\item $2b\neq a$ because otherwise $a$ must be a non-gap.
\end{enumerate}

From the first three conditions we deduce 
$$\max\left(a-g,\left\lceil\frac{g+1}{2}\right\rceil\right)\leq b\leq g.$$
Now, taking also the fourth condition into consideration
we get that the number of options for the pair $a,b$ (and so $n_{g,1}$) is
\begin{eqnarray*}
n_{g,1}&=&\sum_{a=g+1}^{2g+1}(g+1-\max(a-g,\left\lceil\frac{g+1}{2}\right\rceil))-\#\{\mbox{even integers in }\{g+1,\dots,2g+1\}\}\\
&=&\sum_{a=g+1}^{\left\lceil\frac{g+1}{2}\right\rceil+g}(g+1-\left\lceil\frac{g+1}{2}\right\rceil)
+\sum_{\lceil\frac{g+1}{2}\rceil+g+1}^{2g+1}(g+1-a+g)\ \ -\ \ \left\lfloor\frac{g+1}{2}\right\rfloor\\
&=&\sum_{a=g+1}^{\lceil\frac{g+1}{2}\rceil+g}\left\lfloor\frac{g+1}{2}\right\rfloor
+\sum_{\lceil\frac{g+1}{2}\rceil+g+1}^{2g+1}(2g+1-a)\ \ -\ \ \left\lfloor\frac{g+1}{2}\right\rfloor\\
&=&\left\lceil\frac{g+1}{2}\right\rceil\left\lfloor\frac{g+1}{2}\right\rfloor
+\sum_{i=0}^{\lfloor\frac{g-1}{2}\rfloor}i\ \ -\ \ \left\lfloor\frac{g+1}{2}\right\rfloor\\
&=&\left\lceil\frac{g-1}{2}\right\rceil\left\lfloor\frac{g+1}{2}\right\rfloor
+\frac{\lfloor\frac{g-1}{2}\rfloor\lfloor\frac{g+1}{2}\rfloor}{2}.\\
\end{eqnarray*}

\end{proof}

 \begin{corollary}
 For each genus $g\in{\mathbb N}_0$, $n_{g,1}\leq n_{g+1,1}$.
 \end{corollary}

\begin{proof}
If $g$ is even and $g+1$ is odd then
\begin{eqnarray*}
n_{g,1}&=&\frac{3}{8}g^2-\frac{1}{4}g\\
n_{g+1,1}&=&\frac{3}{8}(g+1)^2-\frac{3}{8}=\frac{3}{8}g^2+\frac{3}{4}g\\
\end{eqnarray*}
So, $n_{g+1,1}=n_{g,1}+g\geq n_{g,1}$.

On the other hand, if $g$ is odd and $g+1$ is even then
\begin{eqnarray*}
n_{g,1}&=&\frac{3}{8}g^2-\frac{3}{8}\\
n_{g+1,1}&=&\frac{3}{8}(g+1)^2-\frac{1}{4}(g+1)=\frac{3}{8}g^2+\frac{g}{2}+\frac{1}{8}\\
\end{eqnarray*}
So, $n_{g+1,1}=n_{g,1}+\frac{g+1}{2}\geq n_{g,1}$.
\end{proof}


 \subsection{High ordinarization numbers}
Next we will need Fre\u\i man's Theorem \cite{Freiman1,Freiman2} as formulated
in \cite{Nathanson}.
\begin{theorem}[Fre{\u\i}man]
\label{t:freiman}
Let $A$ be a set of integers such that $\#A=k\geq 3$.
If $\#(A+A)\leq 3k-4$, then $A$ is a subset of an arithmetic 
progression of length $\#(A+A)-k+1\leq 2k-3$.
\end{theorem}

 By means of Fre\u\i man's Theorem we can prove
 the next lemma which tells that the semigroups of large ordinarization number 
 have the first non-gaps even.

 \begin{lemma}
 \label{lemma:parells}
 If a semigroup $\Lambda$ of genus $g$ has ordinarization number $r$ with
 $\frac{g+2}{3}\leq r\leq \lfloor\frac{g}{2}\rfloor$
 then all its non-gaps which are less than or equal to $g$ are even. 
 \end{lemma}

\begin{proof}
Suppose that $\Lambda$ is a semigroup of genus $g$ and ordinarization number 
$r\geq \frac{g+2}{3}$.
This means in particular that 
$\lambda_0=0,\lambda_1,\dots,\lambda_r\leq g$ and
$\lambda_{r+1}\geq g+1$.

Let $A=\Lambda\cap[0,g]=\{\lambda_0,\lambda_1,\dots,\lambda_r\}$. 
We have that the elements in $A+A$ are upper bounded by $2g$ and so 
$A+A\subseteq \Lambda\cap[0,2g]$.
Then $\#(A+A)\leq \#(\Lambda\cap[0,2g])$.
Since the Frobenius number of $\Lambda$ is at most $2g-1$,
$\#(\Lambda\cap[0,2g])=2g+1-g=g+1$.
So, $\#A=r+1$ while $\#(A+A)\leq g+1$.
Now, since $r\geq\frac{g+2}{3}$ we have 
$g+1\leq 3r-1=3(r+1)-4$ and we can apply Theorem~\ref{t:freiman} 
with $k=r+1$.
Thus we have that $A$ is a subset of an arithmetic progression of length at most $g+1-k+1=g-r+1$. Let $d(A)$ be the common difference
of this arithmetic progression.

Now, $d(A)$ can not be larger than or equal to three since otherwise
$\lambda_{r}\geq r\cdot d(A)\geq 3r\geq 3\frac{g+2}{3}>g$, a contradiction with $r$ being the ordinarization number.

If $d(A)=1$ then $A\subseteq[0,g-r]$ and we claim that 
in this case $A\subseteq\{0\}\cup[\lceil\frac{g+1}{2}\rceil,g-r]$.
Indeed, suppose that $x\in A$. Then $2x$ satisfies either $2x\leq g-r$ or $2x\geq g+1$.
If the second inequality is satisfied then it is obvious that 
$x\in\{0\}\cup[\lceil\frac{g+1}{2}\rceil,g-r]$.
If the first inequality is satisfied then we will prove that $mx\leq g-r$ for all $m\geq 2$ 
by induction on $m$ and this leads to $x=0$.
Indeed, if $mx\leq g-r$ then $x\leq \frac{g-r}{m}
\leq \frac{g-\frac{g+2}{3}}{m}
=\frac{2g-2}{3m}<\frac{2g}{3m}$.
Now $(m+1)x< \frac{2g(m+1)}{3m}=\frac{(2m+2)g}{3m}$ 
and since $m\geq 2$, we have $(m+1)x< \frac{(2m+m)g}{3m}=g$.
Since $(m+1)x$ is in $\Lambda\cap[0,g]=A\subseteq[0,g-r]$ this means that
$(m+1)x\leq g-r$ and this proves the claim.

Now, $A\subseteq\{0\}\cup[\lceil\frac{g+1}{2}\rceil,g-r]$
implies that 
$r\leq g-r-\lceil\frac{g+1}{2}\rceil+1=\lfloor\frac{g+1}{2}\rfloor-r\leq 
\frac{g+1}{2}-\frac{g+2}{3}=\frac{g-1}{6}<r$, a contradiction.

So, we deduce that $d(A)=2$, leading to the proof of the lemma.

\end{proof}

\begin{lemma}
\label{lemma:frob}
Suppose that a numerical 
semigroup $\Lambda$ has $\omega$ gaps between $1$ and $n-1$
and $n\geq 2\omega+2$ then 
\begin{enumerate}
\item $n\in\Lambda$,
\item the Frobenius number of $\Lambda$ is smaller than $n$,
\item the genus of $\Lambda$ is $\omega$.
\end{enumerate}
\end{lemma}

\begin{proof}
\begin{enumerate}
\item
The number of pairs $s,t$ with $1\leq s\leq t\leq n-1$ such that $s+t=n$ is $\lfloor\frac{n}{2}\rfloor\geq \lfloor\frac{2\omega+2}{2}\rfloor=\omega+1$.
Since there are $\omega$ gaps between $1$ and $n-1$ this means that in at least one of these pairs both $s$ and $t$ are non-gaps and so $n$ is a sum of two non-gaps and so a non-gap.
\item
Using the same argument one can show by induction that $m\in\Lambda$ for all 
$m\geq n$. Thus, the Frobenius number must be smaller than $n$.
\item
It is a consequence of the previous statements.
\end{enumerate}
\end{proof}

A consequence of this Lemma~\ref{lemma:frob} is the
well known fact that the Frobenius number of a numerical semigroup
of genus $g$ is at most $2g-1$.
Indeed, otherwise take $\omega=g-1$ and $n$ the Frobenius number 
of the semigroup and get a contradiction.

 Let $\Lambda$ be a numerical semigroup.
 We say that a set $B\subset{\mathbb N}_0$ is $\Lambda$-{\it closed}
 if for any $b\in B$ and any $\lambda$ in $\Lambda$, the sum
 $b+\lambda$ is either in $B$ or it is larger than $\max(B)$.
 If $B$ is $\Lambda$-closed so is $B-\min(B)$. Indeed,
 $b-\min(B)+\lambda=(b+\lambda)-\min(B)$ is either in $B-\min(B)$
 or it is larger than $\max(B)-\min(B)=\max(B-\min(B))$.
 The new $\Lambda$-closed set contains $0$.
 We will denote by 
 $C(\Lambda,i)$
 the $\Lambda$-closed sets
 of size $i$ that contain $0$. 

 \begin{theorem}
 \label{theorem:high}
 Let $g\in{\mathbb N}_0$ and let $r$ be an integer with $\frac{g+2}{3}\leq r\leq \lfloor\frac{g}{2}\rfloor$. Define 
 $\omega=\lfloor\frac{g}{2}\rfloor-r$ 
 \begin{enumerate}
 \item
 If $\Omega$ is a numerical semigroup
 of genus $\omega$ and
 $B$ is a $\Omega$-closed set of size $\omega+1$ and first element equal to $0$ then  
 $$\{2j:j \in\Omega\}\cup \{2j-2\max(B)+2g+1:j\in B\}\cup(2g+{\mathbb N}_0)$$ 
 is a numerical semigroup of genus $g$ and ordinarization number $r$.
 \item
 All numerical semigroups of genus $g$ and ordinarization number $r$ 
 can be uniquely written as
 $$\{2j:j \in\Omega\}\cup \{2j-2\max(B)+2g+1:j\in B\}\cup(2g+{\mathbb N}_0)$$ 
 for a unique numerical semigroup
 $\Omega$ of genus $\omega$ 
 and a unique $\Omega$-closed set $B$ of size $\omega+1$ and first element equal to $0$. 
 \item
 The number of numerical semigroups of genus $g$ and ordinarization number $r$
 depends only on $\omega$.
 It is exactly $$\sum_{\mbox{Semigroups }\Omega\mbox{ of genus }\omega}\#C(\Omega,\omega+1).$$
 \end{enumerate}
 \end{theorem}

\begin{proof}
\begin{enumerate}
\item
Suppose that 
$\Omega$ is a numerical semigroup of genus $\omega$ 
and $B$ is a $\Omega$-closed set of size $\omega+1$ and first element equal to $0$. Let $X=\{2j:j \in\Omega\}$, $Y=\{2j-2\max(B)+2g+1:j\in B\}$, $Z=(2g+{\mathbb N}_0)$.

First of all let us see that the complement 
${\mathbb N}_0\setminus(X\cup Y\cup Z)$ 
has $g$ elements. Notice that all elements in $X$ are even while all elements in $Y$ are odd. So, $X$ and $Y$ do not intersect. 
Also the unique element in $Y\cap Z$ is $2g+1$. The number of elements in the complement will be 
\begin{eqnarray*}\#{\mathbb N}_0\setminus(X\cup Y\cup Z)&=&2g-\#\{x\in X:x<2g\}-\#Y+1\\&=&2g-\#\{s\in\Omega:s<g\}-\#B+1\\&=&2g-\omega-\#\{s\in\Omega:s<g\}.
\end{eqnarray*}

We know that all gaps of $\Omega$ are at most $2\omega-1<g$. 
So, $\#\{s\in\Omega:s<g\}=g-\omega$ and we conclude that
$\#{\mathbb N}_0\setminus(X\cup Y\cup Z)=g$.

Before proving that $X\cup Y\cup Z$ 
is a numerical semigroup,
let us prove that the number of non-zero elements in $X\cup Y\cup Z$ 
which are smaller than or equal to $g$ 
is $r$. 
Once we prove that $X\cup Y\cup Z$ is a numerical semigroup,
this will mean, by Lemma~\ref{lemma:equivdef}, 
that it has ordinarization number $r$.
On one hand, 
all elements in $Y$ are larger than $g$. 
Indeed, if $\lambda$ is the enumeration of $\Omega$ then $\max(B)\leq\lambda_\omega\leq 2\omega=
2\lfloor\frac{g}{2}\rfloor-2r\leq
g-2\frac{g+2}{3}<\frac{g}{3} $. 
Now, for any $j\in B$, $2j-2\max(B)+2g+1>2g-2\max(B)>g$. 
On the other hand, 
all gaps of $\Omega$ are at most 
$2\omega-1=2\lfloor\frac{g}{2}\rfloor-2r-1\leq g-\frac{2(g+2)}{3}-1<\frac{g}{3}-1$ and so
all the even integers not belonging to $X$ are less than $g$.
So, the number of non-zero non-gaps of $X\cup Y\cup Z$ smaller than or equal to$g$ is $\lfloor\frac{g}{2}\rfloor-\omega=r$.

To see that $X\cup Y\cup Z$ 
is a numerical semigroup we only need to see that
it is closed under addition. It is obvious that $X+Z\subseteq Z$, $Y+Z\subseteq Z$, $Z+Z\subseteq Z$. It is also obvious that $X+X\subseteq X$ since $\Omega$ is a numerical semigroup and that $Y+Y\subseteq Z$ since, as we proved before, all elements in $Y$ are larger than $g$.

It remains to see that 
$X+Y\subseteq X\cup Y\cup Z$.
Suppose that $x\in X$ and $y\in Y$. Then $x=2i$ for some $i\in\Omega$ and $y=2j-2\max(B)+2g+1$ 
for some $j\in B$. Then $x+y=2(i+j)-2\max(B)+2g+1$. Since $B$ is $\Omega$-closed, we have that either 
$i+j\in B$ and so $x+y\in Y$ or $i+j>\max(B)$. In this case $x+y\in Z$. So, $X+Y\subseteq Y\cup Z$.

\item
First of all notice that,
since the Frobenius number of a semigroup $\Lambda$ of genus $g$ is
smaller than $2g$, it holds
$$\Lambda\cap(2g+{\mathbb N}_0)=(2g+{\mathbb N}_0).$$
For any numerical semigroup
the set $\Omega=\{\frac{j}{2}:j\in\Lambda\cap(2{\mathbb N}_0)\}$ is a numerical semigroup.
If $\Lambda$ has ordinarization number $r\geq\frac{g+2}{3}$ then, 
by Lemma~\ref{lemma:parells},
$$\Lambda\cap[0,g]=(2\Omega)\cap[0,g].$$
The semigroup $\Omega$ has exactly
$r+1$ non-gaps between $0$ and $\lfloor\frac{g}{2}\rfloor$ and
$\omega=\lfloor\frac{g}{2}\rfloor-r$ 
gaps between $0$ and $\lfloor\frac{g}{2}\rfloor$. 
We can use Lemma~\ref{lemma:frob} with
$n=\lfloor\frac{g}{2}+1\rfloor$
since
$$2\omega+2=2\left\lfloor\frac{g}{2}\right\rfloor-2r+2\leq g-\frac{2(g+2)}{3}+2=
\frac{g+2}{3},$$
which implies $2\omega+2\leq \lfloor\frac{g+2}{3}\rfloor\leq \lfloor\frac{g+2}{2}\rfloor=n$.
Then the genus of $\Omega$ is $\omega$ and the Frobenius number of $\Omega$ is 
at most 
$\lfloor\frac{g}{2}\rfloor$.
This means in particular that all even integers larger than $g$
belong to $\Lambda$.

Define $D=(\Lambda\cap[0,2g])\setminus2\Omega$.
That is, $D$ is the set of odd non-gaps of $\Lambda$ smaller than $2g$.
We claim that 
$\bar B=\{\frac{j-1}{2}: j\in D\cup\{2g+1\}\}$
is a $\Omega$-closed set of size $\omega+1$.
The size follows from the fact that the number of non-gaps of $\Lambda$ between 
$g+1$ and $2g$ is $g-r$ and that the number of even integers 
in the same interval is
$\lceil\frac{g}{2}\rceil$.
Suppose that $\lambda\in \Omega$ and $b\in \bar B$.
Then $b=\frac{j-1}{2}$ for some $j$ in $D\cup\{2g+1\}$ and $b+\lambda=
\frac{(j+2\lambda)-1}{2}$.
If $\frac{(j+2\lambda)-1}{2}\geq\max(\bar B)=\frac{(2g+1)-1}{2}$ we are done. Otherwise we have $j+2\lambda\leq 2g$. Since $\Lambda$ is a numerical semigroup and
both $j,2\lambda\in\Lambda$, it holds $j+2\lambda\in\Lambda\cap[0,2g]$.
Furthermore, $j+2\lambda$ is odd since so is $j$. So, $b+\lambda$ is either
larger than $\max(\bar B)$ or it is in $\bar B$.
Then $B=\bar B-\min(\bar B)$ is a $\Lambda$-closed set of size $\omega+1$ and 
first element zero.

\item
It is a consequence of the previous point.
\end{enumerate}
\end{proof}

 Define the sequence $f_\omega$ by 
 $f_\omega=\sum_{\mbox{Semigroups }\Omega\mbox{ of genus }\omega}\#C(\Omega,\omega+1).$
 The first elements in the sequence, from $f_0$ to $f_{\maxomega}$ are
 \begin{center}
 \resizebox{.9\textwidth}{!}{$\begin{array}{|c|ccccccccccccccc|}
 \hline
 \omega & 0 & 1 & 2 & 3 & 4 & 5 & 6 &  7 & 8 & 9 & 10 & 11 & 12 & 13 & 14 \\
 \hline
 f_\omega & 1 & 2 & 7& 23& 68& 200& 615& 1764& 5060& 14626& 41785& 117573& 332475& 933891& 2609832\\\hline\end{array}$}
 \end{center}
 From these first elements and Theorem~\ref{theorem:high} we can deduce $n_{g,r}$
 for any $g$, whenever $r\geq\max(\frac{g+2}{3},\lfloor\frac{g}{2}\rfloor-\maxomega)$. 
This is illustrated in Table~\ref{t:provedresults}.
 Since the sequence $f_\omega$ is increasing for $\omega$ between $0$ and $\maxomega$ we deduce the next corollary.
 \begin{corollary}
 For any $g\in{\mathbb N}$ and 
 any $r\geq\max(\frac{g}{3}+1,\lfloor\frac{g+1}{2}\rfloor-\maxomega)$, it holds $n_{g,r}\geq n_{g+1,r}$.
 \end{corollary}

 If we proved that $f_\omega\leq f_{\omega+1}$ for any $\omega$, 
 this would imply $n_{g,r}\leq n_{g+1,r}$ for any $r>\frac{g}{3}$.

\section{On numerical semigroups with a large number of gap intervals}

In this final section we present a bijection between semigroups at a given depth in the ordinarization tree and semigroups with a given (large) number of gap intervals.

\begin{lemma}
\label{lemma:ordnumb2intnumb}
Suppose that a numerical semigroup $\Lambda$ has genus $g$ and ordinarization number $r\geq \frac{g+2}{3}$. Then it has
$2r$ intervals of gaps if $g$ is even and 
$2r+1$ intervals of gaps if $g$ is odd.
\end{lemma}

\begin{proof}
By Theorem~\ref{theorem:high} we deduce that, defining 
$\omega=\lfloor\frac{g}{2}\rfloor-r$,  $\Lambda$ should be 
$$\{2j:j \in\Omega\}\cup \{2j-2\max(B)+2g+1:j\in B\}\cup(2g+{\mathbb N}_0)$$ 
for a unique numerical semigroup $\Omega$ of genus $\omega$  and a unique $\Omega$-closed set $B$ of size $\omega+1$ and first element equal to $0$. But this semigroup has $g-2\omega$ intervals of gaps.
This number is equal to $g-2\lfloor\frac{g}{2}\rfloor+2r$ which equals $2r$ if $g$ is even and $2r+1$ if $g$ is odd.
\end{proof}

\begin{lemma}
\label{lemma:onbound}
A numerical semigroup with $n$ intervals of gaps has ordinarization number at least $\lfloor\frac{n}{2}\rfloor$.
\end{lemma}

\begin{proof}
The maximum number of gaps 
between $1$ and $g$ is obtained for the semigroup (should it be a semigroup) that has $g-n+1$ non-gaps in a row right after $0$ and then $n-1$ squences of a non-gap and a gap and then no more gaps. In this case there would be $\lceil\frac{n-1}{2}\rceil=\lfloor\frac{n}{2}\rfloor$ 
non-zero non-gaps in the same interval.
\end{proof}

\begin{theorem}
Suppose that a numerical semigroup $\Lambda$ has genus $g$ and $n$ intervals of gaps with $\lfloor\frac{n}{2}\rfloor\geq \frac{g+2}{3}$. Then $g$ and $n$ have the same parity and $\Lambda$ has ordinarization number equal to $\lfloor\frac{n}{2}\rfloor$.
\end{theorem}

\begin{proof}
By Lemma~\ref{lemma:onbound}, $r\geq \frac{g+2}{3}$, and 
Lemma~\ref{lemma:ordnumb2intnumb} gives the result.
\end{proof}

Concluding, if $\lfloor\frac{n}{2}\rfloor\geq \frac{g+2}{3}$, then the set of numerical semigroups of genus $g$ and $n$ intervals of gaps is empty if $n$ and $g$ have different parity and it is exactly the set of numerical semigroups of genus $g$ and ordinarization number $\lfloor\frac{n}{2}\rfloor$ otherwise.

\section*{Acknowledgments}
The author is grateful to Anna de Mier for her suggestions on this work.
This work was partly supported by the Spanish Government through projects TIN2009-11689 ``RIPUP'' and CONSOLIDER INGENIO 2010 CSD2007-00004 ``ARES'',  and by the Government of Catalonia under grant 2009 SGR 1135.

 \begin{table}
 \begin{center}
 \caption{Number of numerical semigroups of each ordinarization number for each genus $g\leq 49$.}
 \label{t:experimentalresults}
\resizebox{.9\textwidth}{!}{\begin{tabular}{c}\begin{tabular}{|c|c|c|c|c|c|c|c|c|c|c|c|c|c|c|c|c|c|c|c|c|c|c|}
\hline \mbox{r$\setminus$ g}&g=0& g=1& g=2& g=3& g=4& g=5& g=6& g=7& g=8& g=9& g=10& g=11& g=12& g=13& g=14& g=15& g=16& g=17& g=18& g=19& g=20& g=21\\\hline
r=0 & {\bf 1} & {\bf 1} & {\bf 1} & {\bf 1} & {\bf 1} & {\bf 1} & {\bf 1} & {\bf 1} & {\bf 1} & {\bf 1} & {\bf 1} & {\bf 1} & {\bf 1} & {\bf 1} & {\bf 1} & {\bf 1} & {\bf 1} & {\bf 1} & {\bf 1} & {\bf 1} & {\bf 1} & {\bf 1} \\r=1 & & & {\bf 1} & {\bf 3} & {\bf 5} & {\bf 9} & {\bf 12} & {\bf 18} & {\bf 22} & {\bf 30} & {\bf 35} & {\bf 45} & {\bf 51} & {\bf 63} & {\bf 70} & {\bf 84} & {\bf 92} & {\bf 108} & {\bf 117} & {\bf 135} & {\bf 145} & {\bf 165} \\r=2 & & & & & {\bf 1} & 2 & 9 & 19 & 39 & 70 & 118 & 196 & 281 & 432 & 586 & 838 & 1080 & 1490 & 1835 & 2449 & 2956 & 3804 \\r=3 & & & & & & & {\bf 1} & {\bf 1} & 4 & 16 & 47 & 97 & 228 & 442 & 844 & 1462 & 2447 & 4017 & 6127 & 9516 & 13693 & 20152 \\r=4 & & & & & & & & & {\bf 1} & {\bf 1} & {\bf 2} & 3 & 28 & 60 & 180 & 442 & 1083 & 2202 & 4611 & 8579 & 15830 & 27493 \\r=5 & & & & & & & & & & & {\bf 1} & {\bf 1} & {\bf 2} & {\bf 2} & 9 & 27 & 93 & 215 & 721 & 1685 & 4417 & 9633 \\r=6 & & & & & & & & & & & & & {\bf 1} & {\bf 1} & {\bf 2} & {\bf 2} & {\bf 7} & 9 & 45 & 89 & 319 & 889 \\r=7 & & & & & & & & & & & & & & & {\bf 1} & {\bf 1} & {\bf 2} & {\bf 2} & {\bf 7} & {\bf 7} & 25 & 47 \\r=8 & & & & & & & & & & & & & & & & & {\bf 1} & {\bf 1} & {\bf 2} & {\bf 2} & {\bf 7} & {\bf 7} \\r=9 & & & & & & & & & & & & & & & & & & & {\bf 1} & {\bf 1} & {\bf 2} & {\bf 2} \\r=10 & & & & & & & & & & & & & & & & & & & & & {\bf 1} & {\bf 1} \\\hline\end{tabular}

\\\\\begin{tabular}{|c|c|c|c|c|c|c|c|c|c|c|c|c|c|c|c|c|}
\hline \mbox{r$\setminus$ g}&g=22& g=23& g=24& g=25& g=26& g=27& g=28& g=29& g=30& g=31& g=32& g=33& g=34& g=35& g=36& g=37\\\hline
r=0 & {\bf 1} & {\bf 1} & {\bf 1} & {\bf 1} & {\bf 1} & {\bf 1} & {\bf 1} & {\bf 1} & {\bf 1} & {\bf 1} & {\bf 1} & {\bf 1} & {\bf 1} & {\bf 1} & {\bf 1} & {\bf 1} \\r=1 & {\bf 176} & {\bf 198} & {\bf 210} & {\bf 234} & {\bf 247} & {\bf 273} & {\bf 287} & {\bf 315} & {\bf 330} & {\bf 360} & {\bf 376} & {\bf 408} & {\bf 425} & {\bf 459} & {\bf 477} & {\bf 513} \\r=2 & 4498 & 5690 & 6582 & 8162 & 9352 & 11370 & 12879 & 15480 & 17317 & 20569 & 22877 & 26812 & 29610 & 34454 & 37739 & 43538 \\r=3 & 27768 & 39726 & 52312 & 72494 & 93341 & 125600 & 157758 & 208370 & 255661 & 331626 & 401389 & 510031 & 608832 & 764927 & 899285 & 1114817 \\r=4 & 46615 & 76616 & 120795 & 189550 & 285103 & 429618 & 618555 & 905721 & 1255646 & 1790138 & 2418323 & 3354611 & 4425179 & 6031518 & 7767784 & 10392180 \\r=5 & 21378 & 41912 & 83951 & 153896 & 281388 & 487211 & 831654 & 1374366 & 2218771 & 3524257 & 5445975 & 8352388 & 12435320 & 18555615 & 26695019 & 38853706 \\r=6 & 2635 & 6446 & 17582 & 39214 & 90574 & 188007 & 394521 & 756910 & 1469758 & 2662254 & 4823002 & 8344482 & 14314198 & 23747986 & 38898550 & 62372773 \\r=7 & 142 & 340 & 1266 & 3483 & 10171 & 26489 & 69692 & 161111 & 382713 & 816457 & 1763299 & 3533977 & 7088495 & 13371197 & 25321828 & 45500820 \\r=8 & {\bf 23} & 24 & 96 & 157 & 553 & 1570 & 5281 & 14835 & 43790 & 113548 & 294908 & 701946 & 1652408 & 3632809 & 7973030 & 16368101 \\r=9 & {\bf 7} & {\bf 7} & {\bf 23} & {\bf 23} & 69 & 95 & 301 & 627 & 2457 & 7168 & 23475 & 68223 & 194677 & 512838 & 1323375 & 3178140 \\r=10 & {\bf 2} & {\bf 2} & {\bf 7} & {\bf 7} & {\bf 23} & {\bf 23} & {\bf 68} & 70 & 228 & 309 & 1142 & 2994 & 10901 & 33846 & 109619 & 318308 \\r=11 & {\bf 1} & {\bf 1} & {\bf 2} & {\bf 2} & {\bf 7} & {\bf 7} & {\bf 23} & {\bf 23} & {\bf 68} & {\bf 68} & 202 & 232 & 740 & 1249 & 4843 & 14332 \\r=12 & & & {\bf 1} & {\bf 1} & {\bf 2} & {\bf 2} & {\bf 7} & {\bf 7} & {\bf 23} & {\bf 23} & {\bf 68} & {\bf 68} & {\bf 200} & 201 & 649 & 759 \\r=13 & & & & & {\bf 1} & {\bf 1} & {\bf 2} & {\bf 2} & {\bf 7} & {\bf 7} & {\bf 23} & {\bf 23} & {\bf 68} & {\bf 68} & {\bf 200} & {\bf 200} \\r=14 & & & & & & & {\bf 1} & {\bf 1} & {\bf 2} & {\bf 2} & {\bf 7} & {\bf 7} & {\bf 23} & {\bf 23} & {\bf 68} & {\bf 68} \\r=15 & & & & & & & & & {\bf 1} & {\bf 1} & {\bf 2} & {\bf 2} & {\bf 7} & {\bf 7} & {\bf 23} & {\bf 23} \\r=16 & & & & & & & & & & & {\bf 1} & {\bf 1} & {\bf 2} & {\bf 2} & {\bf 7} & {\bf 7} \\r=17 & & & & & & & & & & & & & {\bf 1} & {\bf 1} & {\bf 2} & {\bf 2} \\r=18 & & & & & & & & & & & & & & & {\bf 1} & {\bf 1} \\\hline\end{tabular}

\\\\\begin{tabular}{|c|c|c|c|c|c|c|c|c|c|c|c|c|}
\hline \mbox{r$\setminus$ g}&g=38& g=39& g=40& g=41& g=42& g=43& g=44& g=45& g=46& g=47& g=48& g=49\\\hline
r=0 & {\bf 1} & {\bf 1} & {\bf 1} & {\bf 1} & {\bf 1} & {\bf 1} & {\bf 1} & {\bf 1} & {\bf 1} & {\bf 1} & {\bf 1} & {\bf 1} \\r=1 & {\bf 532} & {\bf 570} & {\bf 590} & {\bf 630} & {\bf 651} & {\bf 693} & {\bf 715} & {\bf 759} & {\bf 782} & {\bf 828} & {\bf 852} & {\bf 900} \\r=2 & 47510 & 54320 & 58986 & 67072 & 72419 & 81855 & 88142 & 98946 & 106170 & 118716 & 126844 & 141164 \\r=3 & 1299978 & 1590237 & 1836517 & 2226669 & 2545983 & 3059220 & 3477286 & 4134725 & 4669073 & 5518427 & 6185260 & 7256830 \\r=4 & 13180451 & 17322789 & 21616641 & 28040199 & 34458068 & 44142389 & 53663689 & 67788397 & 81530366 & 102094609 & 121404838 & 150477267 \\r=5 & 54507523 & 77486888 & 106094921 & 148091995 & 198378083 & 272201928 & 358476988 & 483240666 & 626315811 & 833944191 & 1063739070 & 1397557241 \\r=6 & 98298482 & 152816803 & 232801607 & 352797809 & 521753229 & 772496765 & 1114488292 & 1614321267 & 2277566111 & 3242295418 & 4478817624 & 6268430457 \\r=7 & 81612546 & 140878791 & 241699680 & 402445891 & 664483703 & 1072569052 & 1711738040 & 2688862529 & 4165828031 & 6388426599 & 9636305171 & 14462411903 \\r=8 & 33550240 & 65385970 & 126969443 & 235541563 & 436401532 & 777427260 & 1380117648 & 2375549463 & 4064063006 & 6774823275 & 11221522599 & 18200647631 \\r=9 & 7487630 & 16760501 & 36890000 & 77385799 & 160762381 & 319996692 & 631894288 & 1203245544 & 2273796763 & 4158339885 & 7567139870 & 13367227712 \\r=10 & 899807 & 2383461 & 6101724 & 14810797 & 34997273 & 79159902 & 175168573 & 373545010 & 782283651 & 1585487022 & 3171168252 & 6150909456 \\r=11 & 51663 & 164512 & 519339 & 1509557 & 4237829 & 11221868 & 28679326 & 70097864 & 166062233 & 379419480 & 845334246 & 1824208237 \\r=12 & 2527 & 5652 & 21994 & 71261 & 252707 & 803934 & 2492982 & 7226212 & 20114114 & 53281902 & 136131501 & 334153690 \\r=13 & 616 & 649 & 1925 & 2679 & 9947 & 27432 & 106780 & 361575 & 1245778 & 3945659 & 12053243 & 34718395 \\r=14 & {\bf 200} & {\bf 200} & {\bf 615} & 617 & 1800 & 1939 & 6144 & 11138 & 43824 & 140489 & 537134 & 1835716 \\r=15 & {\bf 68} & {\bf 68} & {\bf 200} & {\bf 200} & {\bf 615} & {\bf 615} & 1766 & 1804 & 5254 & 6320 & 22087 & 52194 \\r=16 & {\bf 23} & {\bf 23} & {\bf 68} & {\bf 68} & {\bf 200} & {\bf 200} & {\bf 615} & {\bf 615} & {\bf 1764} & 1765 & 5102 & 5278 \\r=17 & {\bf 7} & {\bf 7} & {\bf 23} & {\bf 23} & {\bf 68} & {\bf 68} & {\bf 200} & {\bf 200} & {\bf 615} & {\bf 615} & {\bf 1764} & {\bf 1764} \\r=18 & {\bf 2} & {\bf 2} & {\bf 7} & {\bf 7} & {\bf 23} & {\bf 23} & {\bf 68} & {\bf 68} & {\bf 200} & {\bf 200} & {\bf 615} & {\bf 615} \\r=19 & {\bf 1} & {\bf 1} & {\bf 2} & {\bf 2} & {\bf 7} & {\bf 7} & {\bf 23} & {\bf 23} & {\bf 68} & {\bf 68} & {\bf 200} & {\bf 200} \\r=20 & & & {\bf 1} & {\bf 1} & {\bf 2} & {\bf 2} & {\bf 7} & {\bf 7} & {\bf 23} & {\bf 23} & {\bf 68} & {\bf 68} \\r=21 & & & & & {\bf 1} & {\bf 1} & {\bf 2} & {\bf 2} & {\bf 7} & {\bf 7} & {\bf 23} & {\bf 23} \\r=22 & & & & & & & {\bf 1} & {\bf 1} & {\bf 2} & {\bf 2} & {\bf 7} & {\bf 7} \\r=23 & & & & & & & & & {\bf 1} & {\bf 1} & {\bf 2} & {\bf 2} \\r=24 & & & & & & & & & & & {\bf 1} & {\bf 1} \\\hline\end{tabular}

\\\\\end{tabular}
}
 \end{center}
 \end{table}

\begin{table}
\caption{Numbers $n_{g,r}$ deduced from Lemma~\ref{lemma:casu} and 
Theorem~\ref{theorem:high}
for each genus $g\leq \maxtaula$.
}
\label{t:provedresults}
\begin{center}
\resizebox{1.3\textwidth}{!}{
\begin{tabular}{l}\begin{tabular}{|c|c|c|c|c|c|c|c|c|c|c|c|c|c|c|c|c|c|c|c|c|c|c|c|c|c|c|c|c|}
\hline \mbox{r$\setminus$ g}&g=50& g=51& g=52& g=53& g=54& g=55& g=56& g=57& g=58& g=59& g=60& g=61& g=62& g=63& g=64& g=65& g=66& g=67& g=68& g=69& g=70& g=71& g=72& g=73& g=74& g=75& g=76& g=77\\\hline
r=0 & {\bf 1}  & {\bf 1}  & {\bf 1}  & {\bf 1}  & {\bf 1}  & {\bf 1}  & {\bf 1}  & {\bf 1}  & {\bf 1}  & {\bf 1}  & {\bf 1}  & {\bf 1}  & {\bf 1}  & {\bf 1}  & {\bf 1}  & {\bf 1}  & {\bf 1}  & {\bf 1}  & {\bf 1}  & {\bf 1}  & {\bf 1}  & {\bf 1}  & {\bf 1}  & {\bf 1}  & {\bf 1}  & {\bf 1}  & {\bf 1}  & {\bf 1} \\r=1 & {\bf 925}  & {\bf 975}  & {\bf 1001}  & {\bf 1053}  & {\bf 1080}  & {\bf 1134}  & {\bf 1162}  & {\bf 1218}  & {\bf 1247}  & {\bf 1305}  & {\bf 1335}  & {\bf 1395}  & {\bf 1426}  & {\bf 1488}  & {\bf 1520}  & {\bf 1584}  & {\bf 1617}  & {\bf 1683}  & {\bf 1717}  & {\bf 1785}  & {\bf 1820}  & {\bf 1890}  & {\bf 1926}  & {\bf 1998}  & {\bf 2035}  & {\bf 2109}  & {\bf 2147}  & {\bf 2223} \\r=2 & ? & ? & ? & ? & ? & ? & ? & ? & ? & ? & ? & ? & ? & ? & ? & ? & ? & ? & ? & ? & ? & ? & ? & ? & ? & ? & ? & ? \\r=3 & ? & ? & ? & ? & ? & ? & ? & ? & ? & ? & ? & ? & ? & ? & ? & ? & ? & ? & ? & ? & ? & ? & ? & ? & ? & ? & ? & ? \\r=4 & ? & ? & ? & ? & ? & ? & ? & ? & ? & ? & ? & ? & ? & ? & ? & ? & ? & ? & ? & ? & ? & ? & ? & ? & ? & ? & ? & ? \\r=5 & ? & ? & ? & ? & ? & ? & ? & ? & ? & ? & ? & ? & ? & ? & ? & ? & ? & ? & ? & ? & ? & ? & ? & ? & ? & ? & ? & ? \\r=6 & ? & ? & ? & ? & ? & ? & ? & ? & ? & ? & ? & ? & ? & ? & ? & ? & ? & ? & ? & ? & ? & ? & ? & ? & ? & ? & ? & ? \\r=7 & ? & ? & ? & ? & ? & ? & ? & ? & ? & ? & ? & ? & ? & ? & ? & ? & ? & ? & ? & ? & ? & ? & ? & ? & ? & ? & ? & ? \\r=8 & ? & ? & ? & ? & ? & ? & ? & ? & ? & ? & ? & ? & ? & ? & ? & ? & ? & ? & ? & ? & ? & ? & ? & ? & ? & ? & ? & ? \\r=9 & ? & ? & ? & ? & ? & ? & ? & ? & ? & ? & ? & ? & ? & ? & ? & ? & ? & ? & ? & ? & ? & ? & ? & ? & ? & ? & ? & ? \\r=10 & ? & ? & ? & ? & ? & ? & ? & ? & ? & ? & ? & ? & ? & ? & ? & ? & ? & ? & ? & ? & ? & ? & ? & ? & ? & ? & ? & ? \\r=11 & ? & ? & ? & ? & ? & ? & ? & ? & ? & ? & ? & ? & ? & ? & ? & ? & ? & ? & ? & ? & ? & ? & ? & ? & ? & ? & ? & ? \\r=12 & ? & ? & ? & ? & ? & ? & ? & ? & ? & ? & ? & ? & ? & ? & ? & ? & ? & ? & ? & ? & ? & ? & ? & ? & ? & ? & ? & ? \\r=13 & ? & ? & ? & ? & ? & ? & ? & ? & ? & ? & ? & ? & ? & ? & ? & ? & ? & ? & ? & ? & ? & ? & ? & ? & ? & ? & ? & ? \\r=14 & ? & ? & ? & ? & ? & ? & ? & ? & ? & ? & ? & ? & ? & ? & ? & ? & ? & ? & ? & ? & ? & ? & ? & ? & ? & ? & ? & ? \\r=15 & ? & ? & ? & ? & ? & ? & ? & ? & ? & ? & ? & ? & ? & ? & ? & ? & ? & ? & ? & ? & ? & ? & ? & ? & ? & ? & ? & ? \\r=16 & ? & ? & ? & ? & ? & ? & ? & ? & ? & ? & ? & ? & ? & ? & ? & ? & ? & ? & ? & ? & ? & ? & ? & ? & ? & ? & ? & ? \\r=17 & ? & ? & ? & ? & ? & ? & ? & ? & ? & ? & ? & ? & ? & ? & ? & ? & ? & ? & ? & ? & ? & ? & ? & ? & ? & ? & ? & ? \\r=18 & {\bf 1764}  & {\bf 1764}  & {\bf 5060} & ? & ? & ? & ? & ? & ? & ? & ? & ? & ? & ? & ? & ? & ? & ? & ? & ? & ? & ? & ? & ? & ? & ? & ? & ? \\r=19 & {\bf 615}  & {\bf 615}  & {\bf 1764}  & {\bf 1764}  & {\bf 5060}  & {\bf 5060} & ? & ? & ? & ? & ? & ? & ? & ? & ? & ? & ? & ? & ? & ? & ? & ? & ? & ? & ? & ? & ? & ? \\r=20 & {\bf 200}  & {\bf 200}  & {\bf 615}  & {\bf 615}  & {\bf 1764}  & {\bf 1764}  & {\bf 5060}  & {\bf 5060}  & {\bf 14626} & ? & ? & ? & ? & ? & ? & ? & ? & ? & ? & ? & ? & ? & ? & ? & ? & ? & ? & ? \\r=21 & {\bf 68}  & {\bf 68}  & {\bf 200}  & {\bf 200}  & {\bf 615}  & {\bf 615}  & {\bf 1764}  & {\bf 1764}  & {\bf 5060}  & {\bf 5060}  & {\bf 14626}  & {\bf 14626} & ? & ? & ? & ? & ? & ? & ? & ? & ? & ? & ? & ? & ? & ? & ? & ? \\r=22 & {\bf 23}  & {\bf 23}  & {\bf 68}  & {\bf 68}  & {\bf 200}  & {\bf 200}  & {\bf 615}  & {\bf 615}  & {\bf 1764}  & {\bf 1764}  & {\bf 5060}  & {\bf 5060}  & {\bf 14626}  & {\bf 14626}  & {\bf 41785} & ? & ? & ? & ? & ? & ? & ? & ? & ? & ? & ? & ? & ? \\r=23 & {\bf 7}  & {\bf 7}  & {\bf 23}  & {\bf 23}  & {\bf 68}  & {\bf 68}  & {\bf 200}  & {\bf 200}  & {\bf 615}  & {\bf 615}  & {\bf 1764}  & {\bf 1764}  & {\bf 5060}  & {\bf 5060}  & {\bf 14626}  & {\bf 14626}  & {\bf 41785}  & {\bf 41785} & ? & ? & ? & ? & ? & ? & ? & ? & ? & ? \\r=24 & {\bf 2}  & {\bf 2}  & {\bf 7}  & {\bf 7}  & {\bf 23}  & {\bf 23}  & {\bf 68}  & {\bf 68}  & {\bf 200}  & {\bf 200}  & {\bf 615}  & {\bf 615}  & {\bf 1764}  & {\bf 1764}  & {\bf 5060}  & {\bf 5060}  & {\bf 14626}  & {\bf 14626}  & {\bf 41785}  & {\bf 41785}  & {\bf 117573} & ? & ? & ? & ? & ? & ? & ? \\r=25 & {\bf 1}  & {\bf 1}  & {\bf 2}  & {\bf 2}  & {\bf 7}  & {\bf 7}  & {\bf 23}  & {\bf 23}  & {\bf 68}  & {\bf 68}  & {\bf 200}  & {\bf 200}  & {\bf 615}  & {\bf 615}  & {\bf 1764}  & {\bf 1764}  & {\bf 5060}  & {\bf 5060}  & {\bf 14626}  & {\bf 14626}  & {\bf 41785}  & {\bf 41785}  & {\bf 117573}  & {\bf 117573} & ? & ? & ? & ? \\r=26 & &  & {\bf 1}  & {\bf 1}  & {\bf 2}  & {\bf 2}  & {\bf 7}  & {\bf 7}  & {\bf 23}  & {\bf 23}  & {\bf 68}  & {\bf 68}  & {\bf 200}  & {\bf 200}  & {\bf 615}  & {\bf 615}  & {\bf 1764}  & {\bf 1764}  & {\bf 5060}  & {\bf 5060}  & {\bf 14626}  & {\bf 14626}  & {\bf 41785}  & {\bf 41785}  & {\bf 117573}  & {\bf 117573}  & {\bf 332475} & ? \\r=27 & & & &  & {\bf 1}  & {\bf 1}  & {\bf 2}  & {\bf 2}  & {\bf 7}  & {\bf 7}  & {\bf 23}  & {\bf 23}  & {\bf 68}  & {\bf 68}  & {\bf 200}  & {\bf 200}  & {\bf 615}  & {\bf 615}  & {\bf 1764}  & {\bf 1764}  & {\bf 5060}  & {\bf 5060}  & {\bf 14626}  & {\bf 14626}  & {\bf 41785}  & {\bf 41785}  & {\bf 117573}  & {\bf 117573} \\r=28 & & & & & &  & {\bf 1}  & {\bf 1}  & {\bf 2}  & {\bf 2}  & {\bf 7}  & {\bf 7}  & {\bf 23}  & {\bf 23}  & {\bf 68}  & {\bf 68}  & {\bf 200}  & {\bf 200}  & {\bf 615}  & {\bf 615}  & {\bf 1764}  & {\bf 1764}  & {\bf 5060}  & {\bf 5060}  & {\bf 14626}  & {\bf 14626}  & {\bf 41785}  & {\bf 41785} \\r=29 & & & & & & & &  & {\bf 1}  & {\bf 1}  & {\bf 2}  & {\bf 2}  & {\bf 7}  & {\bf 7}  & {\bf 23}  & {\bf 23}  & {\bf 68}  & {\bf 68}  & {\bf 200}  & {\bf 200}  & {\bf 615}  & {\bf 615}  & {\bf 1764}  & {\bf 1764}  & {\bf 5060}  & {\bf 5060}  & {\bf 14626}  & {\bf 14626} \\r=30 & & & & & & & & & &  & {\bf 1}  & {\bf 1}  & {\bf 2}  & {\bf 2}  & {\bf 7}  & {\bf 7}  & {\bf 23}  & {\bf 23}  & {\bf 68}  & {\bf 68}  & {\bf 200}  & {\bf 200}  & {\bf 615}  & {\bf 615}  & {\bf 1764}  & {\bf 1764}  & {\bf 5060}  & {\bf 5060} \\r=31 & & & & & & & & & & & &  & {\bf 1}  & {\bf 1}  & {\bf 2}  & {\bf 2}  & {\bf 7}  & {\bf 7}  & {\bf 23}  & {\bf 23}  & {\bf 68}  & {\bf 68}  & {\bf 200}  & {\bf 200}  & {\bf 615}  & {\bf 615}  & {\bf 1764}  & {\bf 1764} \\r=32 & & & & & & & & & & & & & &  & {\bf 1}  & {\bf 1}  & {\bf 2}  & {\bf 2}  & {\bf 7}  & {\bf 7}  & {\bf 23}  & {\bf 23}  & {\bf 68}  & {\bf 68}  & {\bf 200}  & {\bf 200}  & {\bf 615}  & {\bf 615} \\r=33 & & & & & & & & & & & & & & & &  & {\bf 1}  & {\bf 1}  & {\bf 2}  & {\bf 2}  & {\bf 7}  & {\bf 7}  & {\bf 23}  & {\bf 23}  & {\bf 68}  & {\bf 68}  & {\bf 200}  & {\bf 200} \\r=34 & & & & & & & & & & & & & & & & & &  & {\bf 1}  & {\bf 1}  & {\bf 2}  & {\bf 2}  & {\bf 7}  & {\bf 7}  & {\bf 23}  & {\bf 23}  & {\bf 68}  & {\bf 68} \\r=35 & & & & & & & & & & & & & & & & & & & &  & {\bf 1}  & {\bf 1}  & {\bf 2}  & {\bf 2}  & {\bf 7}  & {\bf 7}  & {\bf 23}  & {\bf 23} \\r=36 & & & & & & & & & & & & & & & & & & & & & &  & {\bf 1}  & {\bf 1}  & {\bf 2}  & {\bf 2}  & {\bf 7}  & {\bf 7} \\r=37 & & & & & & & & & & & & & & & & & & & & & & & &  & {\bf 1}  & {\bf 1}  & {\bf 2}  & {\bf 2} \\r=38 & & & & & & & & & & & & & & & & & & & & & & & & & &  & {\bf 1}  & {\bf 1} \\\hline\end{tabular}

\\\\\begin{tabular}{|c|c|c|c|c|c|c|c|c|c|c|c|c|c|c|c|c|c|c|c|c|c|c|c|}
\hline \mbox{r$\setminus$ g}&g=78& g=79& g=80& g=81& g=82& g=83& g=84& g=85& g=86& g=87& g=88& g=89& g=90& g=91& g=92& g=93& g=94& g=95& g=96& g=97& g=98& g=99& g=100\\\hline
r=0 & {\bf 1}  & {\bf 1}  & {\bf 1}  & {\bf 1}  & {\bf 1}  & {\bf 1}  & {\bf 1}  & {\bf 1}  & {\bf 1}  & {\bf 1}  & {\bf 1}  & {\bf 1}  & {\bf 1}  & {\bf 1}  & {\bf 1}  & {\bf 1}  & {\bf 1}  & {\bf 1}  & {\bf 1}  & {\bf 1}  & {\bf 1}  & {\bf 1}  & {\bf 1} \\r=1 & {\bf 2262}  & {\bf 2340}  & {\bf 2380}  & {\bf 2460}  & {\bf 2501}  & {\bf 2583}  & {\bf 2625}  & {\bf 2709}  & {\bf 2752}  & {\bf 2838}  & {\bf 2882}  & {\bf 2970}  & {\bf 3015}  & {\bf 3105}  & {\bf 3151}  & {\bf 3243}  & {\bf 3290}  & {\bf 3384}  & {\bf 3432}  & {\bf 3528}  & {\bf 3577}  & {\bf 3675}  & {\bf 3725} \\r=2 & ? & ? & ? & ? & ? & ? & ? & ? & ? & ? & ? & ? & ? & ? & ? & ? & ? & ? & ? & ? & ? & ? & ? \\r=3 & ? & ? & ? & ? & ? & ? & ? & ? & ? & ? & ? & ? & ? & ? & ? & ? & ? & ? & ? & ? & ? & ? & ? \\r=4 & ? & ? & ? & ? & ? & ? & ? & ? & ? & ? & ? & ? & ? & ? & ? & ? & ? & ? & ? & ? & ? & ? & ? \\r=5 & ? & ? & ? & ? & ? & ? & ? & ? & ? & ? & ? & ? & ? & ? & ? & ? & ? & ? & ? & ? & ? & ? & ? \\r=6 & ? & ? & ? & ? & ? & ? & ? & ? & ? & ? & ? & ? & ? & ? & ? & ? & ? & ? & ? & ? & ? & ? & ? \\r=7 & ? & ? & ? & ? & ? & ? & ? & ? & ? & ? & ? & ? & ? & ? & ? & ? & ? & ? & ? & ? & ? & ? & ? \\r=8 & ? & ? & ? & ? & ? & ? & ? & ? & ? & ? & ? & ? & ? & ? & ? & ? & ? & ? & ? & ? & ? & ? & ? \\r=9 & ? & ? & ? & ? & ? & ? & ? & ? & ? & ? & ? & ? & ? & ? & ? & ? & ? & ? & ? & ? & ? & ? & ? \\r=10 & ? & ? & ? & ? & ? & ? & ? & ? & ? & ? & ? & ? & ? & ? & ? & ? & ? & ? & ? & ? & ? & ? & ? \\r=11 & ? & ? & ? & ? & ? & ? & ? & ? & ? & ? & ? & ? & ? & ? & ? & ? & ? & ? & ? & ? & ? & ? & ? \\r=12 & ? & ? & ? & ? & ? & ? & ? & ? & ? & ? & ? & ? & ? & ? & ? & ? & ? & ? & ? & ? & ? & ? & ? \\r=13 & ? & ? & ? & ? & ? & ? & ? & ? & ? & ? & ? & ? & ? & ? & ? & ? & ? & ? & ? & ? & ? & ? & ? \\r=14 & ? & ? & ? & ? & ? & ? & ? & ? & ? & ? & ? & ? & ? & ? & ? & ? & ? & ? & ? & ? & ? & ? & ? \\r=15 & ? & ? & ? & ? & ? & ? & ? & ? & ? & ? & ? & ? & ? & ? & ? & ? & ? & ? & ? & ? & ? & ? & ? \\r=16 & ? & ? & ? & ? & ? & ? & ? & ? & ? & ? & ? & ? & ? & ? & ? & ? & ? & ? & ? & ? & ? & ? & ? \\r=17 & ? & ? & ? & ? & ? & ? & ? & ? & ? & ? & ? & ? & ? & ? & ? & ? & ? & ? & ? & ? & ? & ? & ? \\r=18 & ? & ? & ? & ? & ? & ? & ? & ? & ? & ? & ? & ? & ? & ? & ? & ? & ? & ? & ? & ? & ? & ? & ? \\r=19 & ? & ? & ? & ? & ? & ? & ? & ? & ? & ? & ? & ? & ? & ? & ? & ? & ? & ? & ? & ? & ? & ? & ? \\r=20 & ? & ? & ? & ? & ? & ? & ? & ? & ? & ? & ? & ? & ? & ? & ? & ? & ? & ? & ? & ? & ? & ? & ? \\r=21 & ? & ? & ? & ? & ? & ? & ? & ? & ? & ? & ? & ? & ? & ? & ? & ? & ? & ? & ? & ? & ? & ? & ? \\r=22 & ? & ? & ? & ? & ? & ? & ? & ? & ? & ? & ? & ? & ? & ? & ? & ? & ? & ? & ? & ? & ? & ? & ? \\r=23 & ? & ? & ? & ? & ? & ? & ? & ? & ? & ? & ? & ? & ? & ? & ? & ? & ? & ? & ? & ? & ? & ? & ? \\r=24 & ? & ? & ? & ? & ? & ? & ? & ? & ? & ? & ? & ? & ? & ? & ? & ? & ? & ? & ? & ? & ? & ? & ? \\r=25 & ? & ? & ? & ? & ? & ? & ? & ? & ? & ? & ? & ? & ? & ? & ? & ? & ? & ? & ? & ? & ? & ? & ? \\r=26 & ? & ? & ? & ? & ? & ? & ? & ? & ? & ? & ? & ? & ? & ? & ? & ? & ? & ? & ? & ? & ? & ? & ? \\r=27 & {\bf 332475}  & {\bf 332475} & ? & ? & ? & ? & ? & ? & ? & ? & ? & ? & ? & ? & ? & ? & ? & ? & ? & ? & ? & ? & ? \\r=28 & {\bf 117573}  & {\bf 117573}  & {\bf 332475}  & {\bf 332475}  & {\bf 933891} & ? & ? & ? & ? & ? & ? & ? & ? & ? & ? & ? & ? & ? & ? & ? & ? & ? & ? \\r=29 & {\bf 41785}  & {\bf 41785}  & {\bf 117573}  & {\bf 117573}  & {\bf 332475}  & {\bf 332475}  & {\bf 933891}  & {\bf 933891} & ? & ? & ? & ? & ? & ? & ? & ? & ? & ? & ? & ? & ? & ? & ? \\r=30 & {\bf 14626}  & {\bf 14626}  & {\bf 41785}  & {\bf 41785}  & {\bf 117573}  & {\bf 117573}  & {\bf 332475}  & {\bf 332475}  & {\bf 933891}  & {\bf 933891}  & {\bf 2609832} & ? & ? & ? & ? & ? & ? & ? & ? & ? & ? & ? & ? \\r=31 & {\bf 5060}  & {\bf 5060}  & {\bf 14626}  & {\bf 14626}  & {\bf 41785}  & {\bf 41785}  & {\bf 117573}  & {\bf 117573}  & {\bf 332475}  & {\bf 332475}  & {\bf 933891}  & {\bf 933891}  & {\bf 2609832}  & {\bf 2609832} & ? & ? & ? & ? & ? & ? & ? & ? & ? \\r=32 & {\bf 1764}  & {\bf 1764}  & {\bf 5060}  & {\bf 5060}  & {\bf 14626}  & {\bf 14626}  & {\bf 41785}  & {\bf 41785}  & {\bf 117573}  & {\bf 117573}  & {\bf 332475}  & {\bf 332475}  & {\bf 933891}  & {\bf 933891}  & {\bf 2609832}  & {\bf 2609832}  & {\bf $f_{15}$} & ? & ? & ? & ? & ? & ? \\r=33 & {\bf 615}  & {\bf 615}  & {\bf 1764}  & {\bf 1764}  & {\bf 5060}  & {\bf 5060}  & {\bf 14626}  & {\bf 14626}  & {\bf 41785}  & {\bf 41785}  & {\bf 117573}  & {\bf 117573}  & {\bf 332475}  & {\bf 332475}  & {\bf 933891}  & {\bf 933891}  & {\bf 2609832}  & {\bf 2609832}  & {\bf $f_{15}$}  & {\bf $f_{15}$} & ? & ? & ? \\r=34 & {\bf 200}  & {\bf 200}  & {\bf 615}  & {\bf 615}  & {\bf 1764}  & {\bf 1764}  & {\bf 5060}  & {\bf 5060}  & {\bf 14626}  & {\bf 14626}  & {\bf 41785}  & {\bf 41785}  & {\bf 117573}  & {\bf 117573}  & {\bf 332475}  & {\bf 332475}  & {\bf 933891}  & {\bf 933891}  & {\bf 2609832}  & {\bf 2609832}  & {\bf $f_{15}$}  & {\bf $f_{15}$}  & {\bf $f_{16}$} \\r=35 & {\bf 68}  & {\bf 68}  & {\bf 200}  & {\bf 200}  & {\bf 615}  & {\bf 615}  & {\bf 1764}  & {\bf 1764}  & {\bf 5060}  & {\bf 5060}  & {\bf 14626}  & {\bf 14626}  & {\bf 41785}  & {\bf 41785}  & {\bf 117573}  & {\bf 117573}  & {\bf 332475}  & {\bf 332475}  & {\bf 933891}  & {\bf 933891}  & {\bf 2609832}  & {\bf 2609832}  & {\bf $f_{15}$} \\r=36 & {\bf 23}  & {\bf 23}  & {\bf 68}  & {\bf 68}  & {\bf 200}  & {\bf 200}  & {\bf 615}  & {\bf 615}  & {\bf 1764}  & {\bf 1764}  & {\bf 5060}  & {\bf 5060}  & {\bf 14626}  & {\bf 14626}  & {\bf 41785}  & {\bf 41785}  & {\bf 117573}  & {\bf 117573}  & {\bf 332475}  & {\bf 332475}  & {\bf 933891}  & {\bf 933891}  & {\bf 2609832} \\r=37 & {\bf 7}  & {\bf 7}  & {\bf 23}  & {\bf 23}  & {\bf 68}  & {\bf 68}  & {\bf 200}  & {\bf 200}  & {\bf 615}  & {\bf 615}  & {\bf 1764}  & {\bf 1764}  & {\bf 5060}  & {\bf 5060}  & {\bf 14626}  & {\bf 14626}  & {\bf 41785}  & {\bf 41785}  & {\bf 117573}  & {\bf 117573}  & {\bf 332475}  & {\bf 332475}  & {\bf 933891} \\r=38 & {\bf 2}  & {\bf 2}  & {\bf 7}  & {\bf 7}  & {\bf 23}  & {\bf 23}  & {\bf 68}  & {\bf 68}  & {\bf 200}  & {\bf 200}  & {\bf 615}  & {\bf 615}  & {\bf 1764}  & {\bf 1764}  & {\bf 5060}  & {\bf 5060}  & {\bf 14626}  & {\bf 14626}  & {\bf 41785}  & {\bf 41785}  & {\bf 117573}  & {\bf 117573}  & {\bf 332475} \\r=39 & {\bf 1}  & {\bf 1}  & {\bf 2}  & {\bf 2}  & {\bf 7}  & {\bf 7}  & {\bf 23}  & {\bf 23}  & {\bf 68}  & {\bf 68}  & {\bf 200}  & {\bf 200}  & {\bf 615}  & {\bf 615}  & {\bf 1764}  & {\bf 1764}  & {\bf 5060}  & {\bf 5060}  & {\bf 14626}  & {\bf 14626}  & {\bf 41785}  & {\bf 41785}  & {\bf 117573} \\r=40 & &  & {\bf 1}  & {\bf 1}  & {\bf 2}  & {\bf 2}  & {\bf 7}  & {\bf 7}  & {\bf 23}  & {\bf 23}  & {\bf 68}  & {\bf 68}  & {\bf 200}  & {\bf 200}  & {\bf 615}  & {\bf 615}  & {\bf 1764}  & {\bf 1764}  & {\bf 5060}  & {\bf 5060}  & {\bf 14626}  & {\bf 14626}  & {\bf 41785} \\r=41 & & & &  & {\bf 1}  & {\bf 1}  & {\bf 2}  & {\bf 2}  & {\bf 7}  & {\bf 7}  & {\bf 23}  & {\bf 23}  & {\bf 68}  & {\bf 68}  & {\bf 200}  & {\bf 200}  & {\bf 615}  & {\bf 615}  & {\bf 1764}  & {\bf 1764}  & {\bf 5060}  & {\bf 5060}  & {\bf 14626} \\r=42 & & & & & &  & {\bf 1}  & {\bf 1}  & {\bf 2}  & {\bf 2}  & {\bf 7}  & {\bf 7}  & {\bf 23}  & {\bf 23}  & {\bf 68}  & {\bf 68}  & {\bf 200}  & {\bf 200}  & {\bf 615}  & {\bf 615}  & {\bf 1764}  & {\bf 1764}  & {\bf 5060} \\r=43 & & & & & & & &  & {\bf 1}  & {\bf 1}  & {\bf 2}  & {\bf 2}  & {\bf 7}  & {\bf 7}  & {\bf 23}  & {\bf 23}  & {\bf 68}  & {\bf 68}  & {\bf 200}  & {\bf 200}  & {\bf 615}  & {\bf 615}  & {\bf 1764} \\r=44 & & & & & & & & & &  & {\bf 1}  & {\bf 1}  & {\bf 2}  & {\bf 2}  & {\bf 7}  & {\bf 7}  & {\bf 23}  & {\bf 23}  & {\bf 68}  & {\bf 68}  & {\bf 200}  & {\bf 200}  & {\bf 615} \\r=45 & & & & & & & & & & & &  & {\bf 1}  & {\bf 1}  & {\bf 2}  & {\bf 2}  & {\bf 7}  & {\bf 7}  & {\bf 23}  & {\bf 23}  & {\bf 68}  & {\bf 68}  & {\bf 200} \\r=46 & & & & & & & & & & & & & &  & {\bf 1}  & {\bf 1}  & {\bf 2}  & {\bf 2}  & {\bf 7}  & {\bf 7}  & {\bf 23}  & {\bf 23}  & {\bf 68} \\r=47 & & & & & & & & & & & & & & & &  & {\bf 1}  & {\bf 1}  & {\bf 2}  & {\bf 2}  & {\bf 7}  & {\bf 7}  & {\bf 23} \\r=48 & & & & & & & & & & & & & & & & & &  & {\bf 1}  & {\bf 1}  & {\bf 2}  & {\bf 2}  & {\bf 7} \\r=49 & & & & & & & & & & & & & & & & & & & &  & {\bf 1}  & {\bf 1}  & {\bf 2} \\r=50 & & & & & & & & & & & & & & & & & & & & & &  & {\bf 1} \\\hline\end{tabular}

\\\\\end{tabular}
}
\end{center}
\end{table}


\end{document}